\documentclass{article}
\usepackage{macros}

\activatecomments

\begin{document}
\title{Free precategories as presheaf categories}
\author[1]{Simon Forest\thanks{The work of this author
    was partially supported by the French ANR project PPS (ANR-19-CE48-0014).}}
\author[2]{Samuel Mimram}
\affil[1]{Aix-Marseille Univ, CNRS, I2M, Marseille, France}
\affil[2]{LIX, École Polytechnique, Palaiseau, France}
\renewcommand\Affilfont{\small}

\maketitle
\begin{abstract}
  Precategories generalize both the notions of strict $n$-category and
  sesquicategory: their definition is essentially the same as the one of strict
  $n$-categories, excepting that we do not require the various interchange laws
  to hold. Those have been proposed as a framework in which one can express
  semi-strict definitions of weak higher categories: in dimension 3, Gray
  categories are an instance of them and have been shown to be equivalent to
  tricategories, and definitions of semi-strict tetracategories have been
  proposed, and used as the basis of proof assistants such as Globular.
  In this article, we are mostly interested in free precategories. Those can be
  presented by generators and relations, using an appropriate variation on the
  notion of polygraph (aka computad), and earlier works have shown that the
  theory of rewriting can be generalized to this setting, enjoying most of the
  fundamental constructions and properties which can be found in the traditional
  theory, contrarily to polygraphs for strict categories.
  We further study here why this is the case, by providing several results which
  show that precategories and their associated polygraphs bear properties which
  ensure that we have a good syntax for those. In particular, we show that the
  category of polygraphs for precategories form a presheaf category.
\end{abstract}


\newpage
\tableofcontents
\newpage

\section*{Introduction}

\paragraph{Strict polygraphs.}
The notion of \emph{polygraph}, also known as \emph{computad}, was introduced by
Street~\cite{street1976limits} and Burroni~\cite{burroni1993higher} as a
generalization of the notion of presentation for strict $n$-categories, thus
extending the now classical notions of presentation for groups and monoids
introduced by Dehn~\cite{dehn1911unendliche} and
Thue~\cite{thue1914probleme}. From an algebraic point of view, they constitute
the right notion of ``free $n$-category'', in the sense that they have been
established as being the cofibrant objects in the folk model structure on the
category of $n$-categories~\cite{metayer2008cofibrant,lafont2010folk}. They thus
allow for computing various invariants of categories, as well as showing
coherence theorems, based on the construction of resolutions (or cofibrant
replacements) of categories of interest. For this reason, one is often
interested in constructing coherent presentations of low-dimensional categories,
which are polygraphs whose underlying free category is suitably equivalent to
the original one.

In order to be able to perform practical computations, one is generally looking
for polygraphs which are as small as possible. This task can be often be
achieved by using techniques originating from rewriting
theory~\cite{baader1999term,bezem2003term}, suitably generalized to this
setting, which exploits the orientation of relations in a presentation. Namely,
when the presentation is terminating and confluent, generators corresponding to
relations between relations can be found as confluence diagrams for critical
branchings. This idea originates in the works of Squier on presented
monoids~\cite{squier1987word,squier1994finiteness,lafont1995new} and has been
the starting point of a series of works exploring higher dimensional
rewriting~\cite{guiraud2009higher, guiraud2012coherence, malbos2013homotopical,
  gaussent2015coherent}, which has since then been further generalized to
various algebraic structures such as term rewriting
systems~\cite{malbos2016homological}, algebras~\cite{guiraud2019convergent} or
operads~\cite{malbos2021completion}.
While polygraphs have thus been proved to be quite a useful tool, they are still
quite unsatisfactory on many aspects.

\paragraph{Limitations of strict polygraphs.}
From a categorical point of view, strict polygraphs are adapted to strict $n$-categories, but
those are known not to be equivalent to weak $n$-categories, which are the real
objects of interest. Namely, already starting from dimension $3$, not every
tricategory is equivalent to a $3$-category: the best we can do is to strictify
associativity and unitality, and show that every tricategory is equivalent to a
Gray category~\cite{gordon1995coherence} (we should underline here that this is
not the only possible partial strictification~\cite{joyal2006weak}). Following
our terminology, a Gray category is a $3$-precategory equipped with interchange
isomorphisms satisfying suitable axioms.
Another categorical defect of polygraphs is the fact that they do not form a
presheaf category. It is namely noted in~\cite{carboni2004corrigenda} that this
cannot be the case because of ``the lack of an ordering'' of $2$-dimensional
(and higher) cells, since composition is commutative for $2$-cells with identity
source and target. More formally, an abstract explanation of the fact that
polygraphs do not form a presheaf category can be found in~\cite{makkai2005word} and
an elementary proof of this fact can be found in~\cite{cheng2012direct}. One route to
solve this consists in restricting to polygraph where generators do not have
identity sources (or targets), which has successfully been explored by
Henry~\cite{henry2018nonunital,henry2018regular}. Our exploration consists here in
taking the other route and ``add ordering'' to morphisms.

From a rewriting point and computer science point of view, polygraphs, when
considered as rewriting systems, lack a fundamental property found in most
settings for rewriting: we expect that a finite rewriting system has a finite number of
critical branchings. This was first observed by Lafont~\cite{lafont2003towards}
and further studied by Guiraud and Malbos who showed that, because of this,
there are finite convergent $3$-polygraphs without finite derivation
type~\cite{guiraud2009higher}.
From a practical perspective, this causes problems. Namely, representing the
possibly infinite families of critical branchings is a difficult challenge, even
in low dimensions~\cite{mimram2014towards}. But in fact, even providing a
concrete representation of morphisms is a challenge, because there is no
canonical representative of morphisms in free categories, up to the axioms of
strict $n$-categories.

\paragraph{Polygraphs for precategories.}
For these reasons, it seems natural to investigate the framework of
\emph{$n$-precategories} whose definition is similar to the one of strict
$n$-categories, excepting that we do not require the interchange laws to
hold. In particular, in dimension~$2$, those correspond to Street's
\emph{sesquicategories}~\cite{street1996categorical}.
We have defined in~\cite{forest2018coherence,forest2021rewriting} an associated
notion of polygraph and developed a theory of rewriting in this setting
(interestingly, Ara{\'u}jo has recently independently come up with a very
similar notion~\cite{araujo2022computads}). It seems that, in this setting, most
of the limitations mentioned above vanish.
First, we now have canonical representatives of morphisms in free
$n$-precategories~\cite{forest2021rewriting}, a property which was first
observed by Makkai while studying strict
$n$-categories~\cite[Section~8]{makkai2005word}, which makes them suitable for
implementing software performing computation on morphisms. For this reason, they
are also used internally in the \emph{Globular} graphical proof
assistant~\cite{vicary2018globular,bar2017data}.
Second, a finite rewriting system has a finite number of critical branchings,
and those can be computed effectively.
Third, we have a hope of being able to deal with weak higher-categories in this
setting. Namely, we have already mentioned that Gray categories are equivalent
to tricategories and are particular $3$-precategories, and putative definitions
of semistrict $4$-categories based on $4$-precategories have been
proposed~\cite{bar2017data}. Note that the polygraph corresponding to a Gray
category is almost never finite, but the infinite families of generators we add
are regular enough to be dealt with in a uniform
way~\cite{forest2018coherence,forest2021rewriting}.

\paragraph{Properties of polygraphs.}
In this article, we further study of the category of $n$-polygraphs for
$n$-precategories. Most importantly, we show that they form a presheaf
category. Our proof is based on the characterization of concrete presheaf
categories given by Makkai~\cite{makkai2005word}. Simultaneously and
independently, another proof of this result has been given by
Ara{\'u}jo~\cite{araujo2022computads}. We should also mention that a notion of
polygraph for weak categories has been developed and shown to be a presheaf
category in~\cite{dean2022computads}. Our approach gives rise to much smaller
polygraphs and thus more amenable computations, although it is not entirely
clear (yet) how to encode weak $n$-categories in our setting, excepting in low
dimensions.

\paragraph{Plan of the paper.}
We begin by introducing precategories and associated polygraphs
(\Cref{sec:precat}) and show that functors between precategories induced by
polygraphs have the important property of being Conduché (\Cref{sec:conduche}),
which is used subsequently. Most of the remainder of the paper is devoted to
showing that polygraphs form a presheaf category. Our proof is based on Makkai's
theorem characterizing presheaf categories (recalled in \Cref{sec:makkai}). In
order to make computations on cells in free precategories, it is useful to
consider their support (\Cref{sec:support}). These allow defining and studying
polyplexes (\Cref{sec:polyplex}) which are shapes parametrizing compositions in
precategories. This finally allows us to show that polygraphs form a presheaf
category (\Cref{sec:presheaf}). As a nice by-product, we derive a parametric
adjunction together with an associated generic-free factorization for
precategories, which gives a more conceptual view of the good syntactical
properties of precategories (\Cref{sec:parametricity}). Finally, we leave two
open questions on homotopical aspects of polygraphs of precategories. First,
whether polygraphs are the cofibrant objects for a reasonable model structure on
precategories (we explain that the usual proof for strict categories does not
immediately generalize to precategories), and second, whether the presheaf
category of polygraphs is able to model homotopy types (we explain why the proof
used by Henry for regular plexes~\cite{henry2018regular} does not adapt here)
(\Cref{sec:folk}).

\section{Precategories and their polygraphs}
\label{sec:precat}
We recall here the definition of $n$\precategories as algebras over globular
sets, as well as their elementary properties. We also recall the associated
notion of polygraph, introduced in earlier
works~\cite{forest2018coherence,forest2021rewriting}, which is a particular
instance of the very general notion of polygraph associated to a monad on
globular sets introduced by Batanin~\cite{batanin1998computads}.

The notion of precategory was first introduced by Street, in dimension $2$,
under the name of \emph{sesquicategory}: this means a
``$1\text{\textonehalf}$-category'', since sesquicategories have more structure
than $1$-categories, but less than $2$-categories (they lack the interchange
law). The general definition of precategories was (implicitly) given by Makkai
in~\cite[Section~8]{makkai2005word}, who used them to deal with the word problem
for free strict categories. Later, they were used as data structures for the
Globular proof assistant~\cite{bar2016globular} and more recently for studying
coherent presentations of Gray categories in \cite{forest2021rewriting} and
coherence for adjunctions~\cite{araujo2022simple,araujo2022coherence}.

In the following, given $n\in\N$, we write $\Nlt n$ for the subset
$\set{0,\ldots,n-1}$ of $\N$, and $\Nle n$ for~$\Nlt{n+1}$.

\paragraph{Globular sets.}
Given~$n\in\N\cup \set\omega$, an \index{globular set}\emph{$n$-globular
  set}~$(X,\csrc,\ctgt)$ (often simply denoted~$X$) is the data of sets~$X_k$
for~$k\in \Nle n$ together with \glossary(dround){$\csrc_i,\ctgt_i$}{the source
  and target operations of a globular set}functions~$\csrc_i,\ctgt_i \co X_{i+1}
\to X_i$ for~$i \in \Nlt{n}$ as in
\[
  \begin{tikzcd}
    X_0
    &
    \ar[l,shift right,"\csrc_0"']
    \ar[l,shift left,"\ctgt_0"]
    X_1
    &
    \ar[l,shift right,"\csrc_1"']
    \ar[l,shift left,"\ctgt_1"]
    X_2
    &
    \ar[l,shift right,"\csrc_2"']
    \ar[l,shift left,"\ctgt_2"]
    \cdots
    &
    \ar[l,shift right,"\csrc_{k-1}"']
    \ar[l,shift left,"\ctgt_{k-1}"]
    X_k
    &
    \ar[l,shift right,"\csrc_{k}"']
    \ar[l,shift left,"\ctgt_{k}"]
    X_{k+1}
    &
    \ar[l,shift right,"\csrc_{k+1}"']
    \ar[l,shift left,"\ctgt_{k+1}"]
    \cdots
  \end{tikzcd}
\]
such that
\begin{align*}
  \csrc_i\circ \csrc_{i+1}&=\csrc_i\circ \ctgt_{i+1}
  &
  \ctgt_i\circ \csrc_{i+1}&=\ctgt_i\circ \ctgt_{i+1}
\end{align*}
for~$i \in \Nlt{n}$.
When there is no ambiguity on~$i$, we often write~$\csrc$ and~$\ctgt$
for~$\csrc_i$ and~$\ctgt_i$ respectively. An element~$u$ of~$X_i$ is called an
\index{globe}\emph{$i$-globe} of~$X$ and, for~$i>0$, the globes~$\csrc_{i-1}(u)$
and~$\ctgt_{i-1}(u)$ are respectively called the
\index{source!of a globe}\emph{source} and
\index{target!of a globe}\emph{target} and~$u$. Given $n$\globular sets~$X$
and~$Y$, a
\index{morphism!of globular sets}\emph{morphism} of $n$\globular sets
between~$X$ and~$Y$ is a family of functions~$F = (F_k\co X_k\to Y_k)_{k \in
  \Nle n}$, such that
\[
  \csrc_i\circ F_{i+1}=F_i\circ \csrc_i
\]
for~$i \in \Nlt{n}$.
We \glossary(Glob){$\nGlob n$}{the category of $n$\globular sets}write~$\nGlob
n$ for the category of $n$\globular sets. We have canonical truncation and
inclusion functors
\[
  \gtruncf_n \co \nGlob{n+1} \to \nGlob n
  \qqand
  \gincf_n \co \nGlob n \to \nGlob{n+1}
\]
which respectively forget the $(n{+}1)$\globes and add an empty set of
$(n{+}1)$\globes. They organize into an adjunction $\gincf_{n+1} \dashv
\gtruncf_n$.
It is direct from definition that globular sets are the models of an
(essentially) algebraic theory, so that the category~$\nGlob n$ is essentially
algebraic. In particular, it implies that it is locally finitely presentable,
complete and cocomplete~\cite{adamek1994locally}.

%

For~$\eps\in\set{-,+}$ and~$j\geq 0$ with $j \le n-i$, we define a morphism
$\csrctgt\eps_{i,j}:X_{i+j}\to X_i$ by
\[
  \csrctgt\eps_{i,j}=\csrctgt\eps_{i}\circ\csrctgt\eps_{i+1}\circ\cdots\circ\csrctgt\eps_{i+j-1}
\]
called the \index{source!iterated}\emph{iterated source} (\resp
\index{target!iterated}\emph{target}) operation when~$\eps=-$ (\resp
~$\eps=+$). We generally omit the index~$j$ when there is no ambiguity and
simply write~$\csrctgt\eps_i(u)$ for~$\csrctgt\eps_{i,j}(u)$.
Given~$i,k,l\in\Nle n$ with~$i < \min(k,l)$, we
\glossary(.abXkxXl){$X_k\times_i X_l$}{the set of pairs of $i$\composable~$k$-
  and $l$\globes}write~$X_k\times_i X_l$ for the pullback
\[
  \begin{tikzcd}
    X_k\times_i X_l \ar[r,dotted] \ar[d,dotted]
    \ar[rd,phantom,"\lrcorner",very near start]& X_l\ar[d,"\csrc_i"]\\
    X_k\ar[r,"\ctgt_i"']& X_i  \pbox.
  \end{tikzcd}
\]
Given~$p\in\N$ and~$k_0,\ldots,k_p \in \Nle n$, a sequence of
globes~$u_0 \in X_{k_0}, \ldots, u_p \in X_{k_p}$ is said
\index{composable}\emph{$i$\composable} for some~$i < \min(k_0,\ldots,k_p)$,
when~$\ctgt_{i}(u_j) = \csrc_i(u_{j+1})$ for~$j \in \Nlt{p}$.
Given~$k \in \Nle n$ and~$u,v \in X_{k}$,~$u$ and~$v$ are said
\index{parallel}\emph{parallel} when~${k = 0}$
or~$\csrctgt\eps_{k-1}(u) = \csrctgt\eps_{k-1}(v)$ for~${\eps \in
  \set{-,+}}$. In order to avoid dealing with the side condition~$k = 0$, we use
the convention that~$X_{-1}$ is the terminal set~$\set{\ast}$ and
that~$\csrc_{-1},\ctgt_{-1}$ are the unique function~$X_0 \to X_{-1}$.

\paragraph{Precategories.}
Given~$n\in\N \cup \set\omega$, an \index{precategory}\emph{$n$\precategory}~$C$
is an $n$\globular set (whose $k$\globes are called \emph{$k$\cells} in this
context) together with, for~$k \in \Nlt{n}$, identity
\glossary(idabpcat){$\unitp k {}$}{the identity operation for
  precategories}operations
\[
  \unitp{k+1} {}\colon C_{k}\to C_{k+1}
\]
for which we use the same notation conventions than the identity operations on
strict categories,
and, for~$k,l \in \N^*_n$, composition \glossary(.cab){$\pcomp_i$}{the
  composition operation for precategories}operations
\[
  \pcomp_{k,l}\co C_k\times_{\min(k,l)-1}C_l\to C_{\max(k,l)}
\]
which satisfy the axioms below.
Given~$i,k,l \in \Nle n$ with~${i = \min(k,l)}$, since the dimensions of the
cells determine the indices of the composition to be used, we often
write~$\pcomp_i$ for~$\pcomp_{k,l}$. In this way, we still make explicit the
most important information which is the dimension~$i$ of composition.
The axioms of $n$\precategories are the following:
\begin{enumerate}[labelwidth=26.26pt,leftmargin=!,label=(P-\roman*),ref=(P-\roman*)]
\item \label{precat:first}\label{precat:src-tgt-unit}for~$k \in \Nlt{n}$ and~$u\in C_k$,
  \[
    \csrc_k(\unitp {k+1}{u})=u=\ctgt_k(\unitp {k+1}{u}),
  \]
\item \label{precat:csrc-tgt}for~$i,k,l \in \Nle n$ such that~$i = \min(k,l) -1$,~$(u,v)\in C_k\times_iC_l$, and~${\eps \in \set{-,+}}$,
  \begin{align*}
    \csrctgt\eps(u \pcomp_i v)
    &=
      \begin{cases}
        u\pcomp_i \csrctgt\eps(v)& \text{if~$k < l$,}\\
        \csrc(u)&\text{if~$k=l$ and~$\eps = -$,}\\
        \ctgt(v)&\text{if~$k=l$ and~$\eps = +$,}\\
        \csrctgt\eps(u)\pcomp_i v&\text{if~$k>l$,}
      \end{cases}
  \end{align*}
\item \label{precat:compat-id-comp}for~$i,k,l \in \Nle n$ with~$i = \min(k,l) -
  1$, given~$(u,v)\in C_{k-1}\times_{i}C_l$,
  \begin{align*}
    \unit u\pcomp_i v
    &=
      \begin{cases}
        v&\text{if~$k \le l$,}\\
        \unit{u\pcomp_i v}&\text{if~$k > l$,}
      \end{cases}
                            \shortintertext{and, given~$(u,v) \in C_k \times_i C_{l-1}$,}
                            u\pcomp_i\unit{v}
         &=
           \begin{cases}
             u&\text{if~$l\le k$,}\\
             \unit{u\pcomp_i v}&\text{if~$l > k$,}
           \end{cases}
  \end{align*}

\item \label{precat:before-last} \label{precat:assoc}for~$i,k,l,m \in \Nle n$ with~$i = \min(k,l) - 1 =
  \min(l,m) - 1$, and~$u \in C_k$,~$v \in C_l$ and~$w \in C_w$ such that~$u,v,w$
  are $i$\composable,
  \[
    (u\pcomp_iv)\pcomp_iw
    =
    u\pcomp_i(v\pcomp_iw),
  \]
\item \label{precat:distrib}\label{precat:last}for~$i,j,k,l,l' \in \Nle n$ such that
  \[
    i = \min(k,\max(l,l')) - 1,
    \qquad
    j = \min(l,l') - 1
    \qqtand
    i < j\zbox,
  \]
  given~$u \in C_k$ and~$(v,v') \in C_l \times_j C_{l'}$ such that~$u,v$ are $i$\composable,
  \[
    u \pcomp_i (v \pcomp_j v') = (u \pcomp_i v) \pcomp_j (u \pcomp_i v')
  \]
  and, given~$(u,u') \in C_l \times_j C_{l'}$ and~$v \in C_k$ such that~$u,v$
  are $i$\composable,
  \[
    (u \pcomp_j u') \pcomp_i v = (u \pcomp_i v) \pcomp_j (u' \pcomp_i v).
  \]
\end{enumerate}

Note that, provided that the Axioms~\ref{precat:first}
to~\ref{precat:before-last} are satisfied, \Axr{precat:distrib} can be shown
equivalent to the more symmetrical axiom
\begin{enumerate}[labelwidth=26.26pt,leftmargin=!,label=\textup{\ref{precat:distrib}'},ref=\textup{\ref{precat:distrib}'}]
\item for every~$i,j,k \in \Nle{n}$ satisfying~$i < j < k$, and
  cells~$u_1,u_2 \in C_{i+1}$,~$v_1,v_2 \in C_{j+1}$ and~$w \in C_k$ such
  that~$u_1,w,u_2$ are $i$\composable and~$v_1,w,v_2$ are $j$\composable, we
  have
  \[
    u_1 \pcomp_i (v_1 \pcomp_j w \pcomp_j v_2) \pcomp_i  u_2 = (u_1 \pcomp_i
    v_1 \pcomp_i u_2) \pcomp_j (u_1
    \pcomp_i w \pcomp_i u_2) \pcomp_j (u_1 \pcomp_i v_2 \pcomp_i u_2)\zbox.
  \]
\end{enumerate}

\begin{example}
  Given a $2$\precategory~$C$ with two $2$\cells~$\phi$ and~$\psi$ as in
  
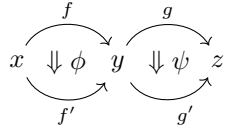
\[
    \begin{tikzcd}
      x
      \ar[r,bend left=60,"f",""{auto=false,name=topl}]
      \ar[r,bend right=60,"{f'}"'{pos=0.52},""{auto=false,name=botl}]
      &
      y
      \ar[r,bend left=60,"g",""{auto=false,name=topr}]
      \ar[r,bend right=60,"{g'}"',""{auto=false,name=botr}]
      &
      z
      \ar[from=topl,to=botl,phantom,"\Downarrow\phi"]
      \ar[from=topr,to=botr,phantom,"\Downarrow\psi"]
    \end{tikzcd}
  \]
  there are two ways to compose~$\phi$ and~$\psi$ together, given by
  \[
    (\phi \pcomp_0 g) \pcomp_1 (f' \pcomp_0 \psi)
    \qtand
    (f \pcomp_0 \psi) \pcomp_1 (\phi \pcomp_0 g')
  \]
  that can be represented using string diagrams by
  \[
    \satex{wires-left-right}
    \quad\qqtand\quad
    \satex{wires-right-left}
  \]
  and these two composites are not expected to be equal in~$C$. Moreover, by our
  definition of precategories, there is no such thing as a valid cell~$\phi
  \pcomp_0 \psi$, and the string diagram
  \[
    \satex{wires-mid-mid}
  \]
  makes no sense in this setting.
\end{example}

Given two $n$\precategories~$C$ and~$D$, a \index{morphism!of
  precategories}\emph{morphism} of $n$-precategories (or
\index{prefunctor@$n$\prefunctor}\emph{$n$\prefunctor}) between~$C$ and~$D$ is a
morphism of $n$\globular sets~$F\co C \to D$ such that
\begin{itemize}
\item $F(\unitp {k+1} u) = \unitp {k+1} {F(u)}$ for~$k \in \Nlt{n}$ and~$u \in C_k$,
  
\item $F(u \pcomp_i v) = F(u) \pcomp_i F(v)$ for~$i,k,l \in \Nle n$ with~$i =
  \min(k,l) - 1$ and~${(u,v) \in C_k \times_i C_l}$.
\end{itemize}
We \glossary(PCatn){$\nPCat n$}{the category of $n$\precategories}write~$\nPCat
n$ for the category of $n$\precategories thus defined. We have canonical
truncation and inclusion functors
\[
  \pctruncf_n \co \nPCat{n+1} \to \nPCat n
  \qqand
  \pcincf_n \co \nPCat n \to \nPCat{n+1}
\]
which respectively forget the $(n{+}1)$\cells and add a set of $(n{+}1)$\cells
consisting of formal identities of $n$\cells. They organize into an adjunction
$\pcincf_{n+1} \dashv \pctruncf_n$.

\paragraph{The globular monad of $n$-precategories.}
The above definition of $n$\precategories directly translates into an
essentially algebraic theory so that the category $\nPCat n$ is locally finitely
presentable~\cite{adamek1994locally}.
There is a forgetful functor
\[
  \ccfgff_n \co \nPCat n \to \nGlob n
\]
which maps an $n$\precategory to its underlying $n$\globular set, and this
functor is induced by the inclusion of the essentially algebraic theory of
$n$-globular sets into the one of $n$-precategories. We thus have the
following~\cite[Proposition~1.4.2.4]{forest:tel-03155192}:

\begin{prop}
  \label{prop:pcatn-complete-cocomplete-ra}
  The category~$\nPCat n$ is locally finitely presentable, complete and
  cocomplete. Moreover, the functor~$\ccfgff_n$ is a right adjoint which
  preserves directed colimits.
\end{prop}

\noindent
The above proposition states the existence of a functor
\[
  \freealgf_n \co \nGlob n \to \nPCat n
\]
which is left adjoint to $\ccfgff_n$, sending an $n$-globular set to the
$n$-precategory it freely generates. Moreover, the functor~$\ccfgff_n$ can be
shown monadic using Beck's monadicity
theorem~\cite[Proposition~1.4.2.5]{forest:tel-03155192}:

\begin{prop}
  \label{prop:pc-cat-fgf-monadic}
  For every~$n \in \N\cup\set\omega$, the functor~$\ccfgff_n$ is monadic.
\end{prop}

\noindent
This shows that, for $n\in \Ninf$, $\nPCat n$ is the category of algebras for a
monad $T^n \co \nGlob n \to \nGlob n$ on $n$\globular sets (the monad induced by
the above adjunction).

\paragraph{Polygraphs of precategories.}
In fact, for $n \in \N$, the monads $T^n$ is adequately derived by truncation
from $T^\omega$~\cite[Theorem~1.4.2.8]{forest:tel-03155192}, the latter being
\emph{truncable} in the sense of Batatnin~\cite{batanin1998computads}. By
general arguments on globular algebras, this allows the definition of
\emph{polygraphs} for the theory of precategories.

The category of \emph{$n$\polygraphs} $\nPol n$ (for $n$-precategories) is
defined by induction on $n$, together with a functor
\[
  \freecat[n]- \co \nPol n \to \nPCat n
\]
often written $\freecat-$, which associates to an $n$-polygraph the
$n$-precategory it freely generates, as follows.
We first define $\nPol 0 = \nGlob 0$ (which is isomorphic to $\Set$) and
$\freecat[0]- = \freealgf_0$ (which is the identity functor on~$\Set$). Now,
given $n \in \N$, assuming $\nPol n$ and $\freecat[n]-$ defined in dimension
$n$, we define $\nPol{n+1}$ as the pullback
\[
  \begin{tikzcd}[column sep=12mm]
    \nPol{n+1}
    \ar[r,dashed,"\poltoglob_{n+1}"]
    \ar[d,dashed,"\poltruncf_n"']
    \ar[phantom,dr,very near start,"\lrcorner"]
    &
    \nGlob{n+1}
    \ar[d,"\gtruncf_n"]
    \\
    \nPol n
    \ar[r,"{\fgfalgf_n \freecat[n]-}"']
    &
    \nGlob n
  \end{tikzcd}
\]
The functor $\poltruncf_n:\nPol{n+1}\to\nPol{n}$, called the
\emph{$n$-truncation} functor for polygraphs, admits a left adjoint
$\polincf_{n+1}:\nPol{n}\to\nPol{n+1}$, which extends an $n$-polygraph~$\P$ as
an $(n{+}1)$-polygraph with an empty set of $(n{+}1)$-generators (using the
description of polygraphs given just below).
The image $\freecat\P$ under $\freecat[n+1]-$ of an $(n{+}1)$\polygraph~$\P$ is
defined as the pushout
\[
  \begin{tikzcd}[column sep={16em,between origins}]
    \freealgf_{n+1}\gincf_{n+1} \gtruncf_n \poltoglob_{n+1}\P
    \ar[r,"(\freealgf_{n+1} \gtrunccu_n\poltoglob_{n+1})_{\P}"]
    \ar[d,"\alpha_\P"']
    &
    \freealgf_{n+1}\poltoglob_{n+1}\P
    \ar[d,dashed]
    \\
    \pcincf_{n+1}\freecat {(\poltruncf_n\P)}
    \ar[r,dashed]
    &
    \freecat\P
    \phar[lu,"\ulcorner",very near start]
  \end{tikzcd}
\]
where $\gtrunccu_n$ is the counit of the adjunction
$\gincf_{n+1} \dashv \gtruncf_n$ and $\alpha_\P$ is the composite
\[
  \alpha_\P
  \qeq 
  \begin{tikzcd}[column sep=5em]
    \freealgf_{n+1}\gincf_{n+1} \gtruncf_n \poltoglob_{n+1}\P
    \ar[r,"\sim"]
    &[-2em]
    \pcincf_{n+1}\freealgf_{n} \fgfalgf_n \freecat[n]-\poltruncf_{n}\P
    \ar[r,"{(\pcincf_{n+1}\varepsilon_n \freecat[n]-\poltruncf_{n})_{\P}}"]
    &[2em]
    \pcincf_{n+1}\freecat {(\poltruncf_n\P)}
  \end{tikzcd}
\]
where $\varepsilon_n$ is the counit of the adjunction $\freealgf_n \dashv \fgfalgf_n$.
Intuitively, $\freecat \P$ is obtained by freely generating an
$(n{+}1)$\precategory from $\freecat{(\poltruncf_n \P)}$ by attaching the
$(n{+}1)$\generators described by $\poltoglob_{n+1}\P$.
The mapping $\P \mapsto \freecat \P$ then naturally extends to a functor
$\freecat[n]- \co \nPol {n+1} \to \nPCat n$, which concludes the inductive
definition of polygraphs of precategories. More details on this construction can
be found in~\cite{forest:tel-03155192,forest2021rewriting}.

Since the monad of the theory of precategory is truncable, given~$n \in \N$, an
$n$\polygraph~$\P$ can be alternatively described as a diagram in~$\Set$ of the
form
\[
  \begin{tikzcd}[column sep=10ex,labels={inner sep=0.5pt}]
    \P_0\ar[d,"\polinj0"{inner sep=2pt}]
    &\P_1
    \ar[dl,shift right,"\gsrc_0"',pos=0.3]
    \ar[dl,shift left,"\gsrc_0",pos=0.3]\ar[d,"\polinj1"{inner sep=2pt}]
    &\P_2\ar[dl,shift right,"\gsrc_1"',pos=0.3]\ar[dl,shift left,"\gsrc_1",pos=0.3]\ar[d,"\polinj2"{inner sep=2pt}]
    &\ldots
    &\P_{n-1}
    \ar[dl,shift right,"\gsrc_{n-2}"',pos=0.3]
    \ar[dl,shift
    left,"\gsrc_{n-2}",pos=0.3]\ar[d,"\polinj{n}"{inner sep=2pt}]
    &\P_{n}\ar[dl,shift right,"\gsrc_{n-1}"',pos=0.3]\ar[dl,shift left,"\gsrc_{n-1}",pos=0.3]\\
    \freecat{\P_0}
    &
    \freecat{\P_1}
    \ar[l,shift right,"{\csrc_0}"']
    \ar[l,shift
    left,"{\ctgt_0}"]
    &\ldots
    \ar[l,shift
    right,"{\csrc_1}"']\ar[l,shift
    left,"{\ctgt_1}"]
    &\freecat{\P_{n-2}}
    &\ar[l,shift
    right,"{\csrc_{n-2}}"']\ar[l,shift
    left,"{\ctgt_{n-2}}"]\freecat{\P_{n-1}}
  \end{tikzcd}
\]
where, for~$i \in \Nlt{n}$,~$\polinj i$ is the embedding of the $i$\generators
$\P_i$ into the set~$\freecat{\P_i}$ of freely generated $i$\cells, such that
\[
  {\csrc_i}\circ\gsrc_{i+1}={\csrc_i}\circ\gtgt_{i+1}
  \qqtand
  \ctgt_i\circ\gsrc_{i+1}=\ctgt_i\circ\gtgt_{i+1}
\]
for~$i \in \Nlt{n}$. Note that the above description is the same as the original
definition of polygraphs by Burroni~\cite{burroni1993higher}, excepting that the
sets $\freecat{\P_i}$ of $i$~cells are freely generated as $i$-precategories
instead of strict $i$-categories.

By general properties on locally presentable categories, we have:
\begin{prop}
  \label{prop:npol-loc-fin-pres}
  Given $n \in \Ninf$, $\nPol n$ is a locally finitely presentable category. In
  particular, it is complete and cocomplete.
\end{prop}
\begin{proof}
  The $2$\category of locally presentable categories, right adjoints (\resp left
  adjoints) and natural transformations is closed under bipullbacks (see
  \cite[Theorem 2.18,Theorem 3.15]{bird1984thesis}). A pullback along an
  isofibration happens to be a bipullback and the pullback of $\gtruncf_n$ along
  $\fgfalgf_n$ can be shown to be a left adjoint and again an isofibration.
  Then, its pullback by $\freecat[n]-$, which is known (see \cite[Proposition
  3.1]{batanin1998computads}) to be a left adjoint, is again a left adjoint
  whose domain $\nPol n$ is a locally presentable category. A more detailed
  study shows that $\nPol n$ is locally \emph{finitely} presentable with finite
  polygraphs as finitely presentable objects. See \cite[Proposition
  3.3.3]{forest2022extension} for the local presentability and \cite[Theorem
  1.3.3.19]{forest:tel-03155192} for the local \emph{finite} presentability.
\end{proof}
\noindent In the following, we will write $\polterm$ for the terminal object of $\nPol n$,
for $n \in \Ninf$.


\section{Free functors are Conduché}
\label{sec:conduche}
Free precategories on polygraphs enjoy useful properties, thanks to which we
have a nice syntax for morphisms in those, as we now show. It should be noted
that many those are not valid in the usual setting of polygraph for strict
categories (as opposed to precategories).
One remarkable such property of free precategories is that their cells can be
described as canonical compositions of generators, which happen to be unique for
a given cell, so that we prefer to call them \emph{normal forms}. These normal
forms are adequately reflected by free functors, since the latter reflect
elementary compositions: in other words, they satisfy the analogue of the
Conduché property for strict categories~\cite{guetta2020polygraphs}.
In addition to providing convenient tools in the proofs, we will see in
subsequent sections that these properties entail the existence universal shapes
of compositions.

\paragraph*{Types and contexts.}
Given $m \le n \in \N$, an $n$\precategory $C$, an \emph{$m$\type} is a pair of
parallel $(m{-}1)$\cells of $C$. We use the convention that there is a unique
$0$\type, and all pairs of $0$\cells of a precategory are parallel. Given a
$k$\cell $u \in C$ for some $k \ge m$, $u$ has a canonical associated $m$\type:
$(\csrc_{m-1}(u),\ctgt_{m-1)}(u))$. In the following, an $m$\type is thought of
as the type for a formal variable, which suggests defining the notion of context
(a morphism in which the variable occurs exactly once) and of substitution
(replacing the variable by a morphism).

An \emph{$m$\context} $E$ for an $m$\type $(s,t)$ is defined by induction
on~$m$, together with the \emph{evaluation}~$E[u]$ of $E$ at a cell of $m$\type
$(s,t)$:
\begin{itemize}
\item there is a unique $0$\context of the unique $0$\type, denoted $[-]$, and
  the evaluation of it at a cell~$u \in C$ is $u$,
\item an $(m{+}1)$\context of type $(s,t)$ is a triple $E = (l,E',r)$ with
  $l,r \in C_m$, and $E'$ an $m$\context of type
  $(\csrc_{m-1}(s),\ctgt_{m-1}(t))$ such that $\ctgt_{m}(l) = E'[s]$ and
  $E'[t] = \csrc_m(r)$, and the evaluation $E[u]$ of $E$ at a cell $u$ is
  defined by $E[u] = l \comp_m E'[u] \comp_m r$.
\end{itemize}
Alternatively, an $m$\context $E$ can be thought of as an expression of the form
\[
  l_m \comp_{m-1} ( \cdots \comp_1 (l_1 \comp_0[-]\comp_0 r_1) \comp_1 \cdots ) \comp_{m-1} r_m
\]
where the $l_i,r_i \in C_i$ are the $i$\cells occurring in the definition of $E$
for $i \in \Nle m$, and its evaluation at a cell~$u$ as the cell obtained by
replacing $[-]$ by~$u$ in the above expression.

\paragraph{Normal forms.}
We have the following normal form for the cells of free precategories:
\begin{theo}
  \label{thm:canonical-form}
  Given $m \in \N$
  and a polygraph $\P\in\oPol$, every $m$\cell of $\freecat\P$ can be written
  uniquely as
  \[
    E_1[g_1] \comp_{m-1} \cdots \comp_{m-1} E_k[g_k]
  \]
  for some unique $g_1,\cdots,g_k \in \P_m$ and $(m{-}1)$\contexts
  $E_1,\ldots,E_k$ of the corresponding types.
\end{theo}
\begin{proof}
  We only sketch the proof, which is detailed
  in~\cite[Theorem~1.8.3]{forest2021rewriting}. One can adequately orient the
  axioms \ref{precat:first}--\ref{precat:last} of precategories in order to
  obtain a terminating and locally confluent rewriting system on the formal
  expressions of cells of free precategories. By standard arguments of rewriting
  theory~\cite{baader1999term}, this gives the existence and unicity of normal
  forms.
\end{proof}
\begin{rem}
  \label{rem:polinj-injective}
  A consequence of the above theorem is that the embeddings $\polinj i \co \P_i
  \to \freecat\P_i$ introduced earlier are injective. Thus, given $g \in \P_i$,
  we will often omit $\polinj i$ and write $g$ for both the element of $\P_i$
  and the cell of $\freecat\P_i$.
\end{rem}

\noindent
The unicity of normal forms directly entails the that the image under a free
functor of an identity is an identity (\resp of a generator is a generator):

\begin{prop}
  \label{prop:freecat-functors}
  Let~$F\co \P \to \Q \in \oPol$ be a morphism of polygraphs,~$k \in \Nle n$
  and~$u \in \freecat\P_k$. The following hold:
  \begin{enumerate}[label=(\roman*),ref=(\roman*)]
  \item \label{prop:freecat-functors:unit} when~$k > 0$, there exists a cell~$u'
    \in \freecat\P_{k-1}$ such that~$u = \unitp k {u'}$ if and only if there
    exists a cell~$\tilde u' \in \freecat\Q_{k-1}$ such that~$\freecat F(u) = \unitp k
    {\tilde u'}$,
  \item \label{prop:freecat-functors:gen} there exists a generator~$g \in \P_k$
    such that~$u = g$ if and only if there exists a generator~$\tilde g \in
    \Q_k$ such that~$\freecat F(u) = \tilde g$.
  \end{enumerate}
\end{prop}

We should also mention now that composition in free precategories is
cancellative. This does not seem to be deducible from the more general
properties developed in the next sections.

\begin{prop}
  \label{lem:pcomp-inj}
  Given $\P \in \oPol$ and $u,v_1,v_2 \in \freecat \P$ such that $u \pcomp_i v_1
  = u \pcomp_i v_2$ for some $i$, then~$v_1 = v_2$.
\end{prop}
\begin{proof}
  Note that, by the input and output dimension conditions of $\pcomp_i$, we
  necessarily have that the dimension of $v_1$ is the one of $v_2$. We do an
  induction on the dimension of the resulting cell $u \comp_i v_1$ and
  distinguish three cases depending on the relative dimensions of $u$, $v_1$ and
  $v_2$.
  \begin{itemize}
  \item Suppose that $u,v_1,v_2\in \freecat\P_{i+1}$. By unicity of the
    decomposition of $(i{+}1)$\cells of free precategories
    (\Cref{thm:canonical-form}) and its compatibility with $i$\composition as
    concatenation, we have~$v_1 = v_2$.
  \item Suppose $u \in \freecat\P_{i+1}$ and $v_1,v_2 \in \freecat\P_{n}$ with
    $n>i+1$. We reason by induction on~$v_1$.
    \begin{itemize}
    \item Suppose that $v_1 = \alpha$ for some generator $\alpha \in
      \P_n$. Then, by the definition of composition and the normal forms, we
      have that $v_2 = \alpha$.
    \item Suppose that $v_1 = E_1[\alpha]$ for some generator $\alpha \in \P_n$
      and $m$\context $E_1$ with $0 < m < n$. By the definition of composition
      and the unicity of normal forms, we have $v_2 = E_2[\alpha]$. Let
      $(l_j,E'_j,r_j) = E_j$ for $j \in \set{1,2}$. If $m = i+1$, then, by
      unicity of normal forms, we have $u \pcomp_i l_1 = u \pcomp_i l_2$,
      $E'_1[\alpha] = E'_2[\alpha]$ and $r_1 = r_2$. By the beginning of the
      proof, we have $l_1 = l_2$, so that $v_1 = v_2$. Otherwise, if $m > i+1$,
      then $u \pcomp_i l_1 = u \pcomp_i l_2$, $u \pcomp_i E'_1[\alpha] = u
      \pcomp_i E'_2[\alpha]$ and $u \pcomp_i r_1 = u \pcomp_i r_2$. By the
      different induction hypotheses, we have $l_1 = l_2$, $E'_1[\alpha] =
      E'_2[\alpha]$ and $r_1 = r_2$, so that $v_1 = v_2$.
    \item If
      $v_1 = E^1_1[\alpha_{1}] \pcomp_{n-1} \cdots \pcomp_{n-1}
      E^k_1[\alpha_{k}]$, then we necessarily have
      $v_2 = E^k_2[\alpha_1] \pcomp_{n-1} \cdots \pcomp_{n-1} E^k_2[\alpha_k]$
      such that $u \pcomp_i E^j_1[\alpha_j] = u \pcomp_i E^j_2[\alpha_j]$. By
      the previous argument, we have $E^j_1[\alpha_j] = E^j_2[\alpha_j]$, so
      that $v_1 = v_2$.
    \end{itemize}
  \item Suppose that $u \in \freecat\P_{n}$, $v_1,v_2 \in \freecat\P_{i+1}$ with
    $n > i+1$. We reason by induction on $u$.
    \begin{itemize}
    \item If $u = \unit {u'}$, then we have $u' \pcomp_i v_1 = u' \pcomp_i v_2$,
      so that $v_1 = v_2$ by induction.
    \item If $u = E[\alpha]$ for some $(n{-}1)$\context $E$, then let
      $(l,E',r)=E$. We then have $r \pcomp_i v_1 = r \pcomp_i v_2$ so that, by
      induction hypothesis, $v_1 = v_2$.
    \item If $u = E_1[\alpha_1] \comp_{n-1} \cdots \comp_{n-1} E_k[\alpha_k]$
      for some $k \ge 1$, $\alpha_1,\ldots,\alpha_k \in \P_n$ and contexts
      $E_1,\ldots,E_k$, then we have in particular
      $E_1[\alpha_1] \comp_i v_1 = E_1[\alpha_1] \comp_i v_2$ so that we can
      conclude $v_1 = v_2$ by the previous case.
    \end{itemize}
    
  \item Suppose that $u,v_1,v_2 \in \freecat\P_{i+1}$, $v_1,v_2 \in
    \freecat\P_{i+1}$ with $n > 0$. By the unicity of normal forms, we can
    uniquely write $u$ as $E^1[\alpha^1] \comp_i \cdots \comp_i E^k[\alpha^k]$
    and $v_j$ as $E^1_j[\beta^1_j] \comp_i \cdots \comp_i
    E^{l_j}_j[\beta^{l_j}_j]$ for $j \in \set{1,2}$ for some adequate $k,l_1,l_2
    \in \N$, $i$\contexts $E^{\cdot}$, $E^{\cdot}_1$, $E^{\cdot}_2$ and
    $(i{+}1)$\generators $\alpha^{\cdot}$, $\beta^\cdot_1$ and
    $\beta^\cdot_2$. By considering the induced normal forms on $u \comp_i
    v_1$ and $u \comp_i v_2$ by concatenation, we deduce by unicity of normal
    forms that $l_1 = l_2$ and $E^\cdot_1 = E^\cdot_2$ and $\beta^\cdot_1
    = \beta^\cdot_2$, so that $v_1 = v_2$.\qedhere
  \end{itemize}
\end{proof}
\begin{rem}
  \label{rem:sc-non-cancellative}
  Note that such a property does not hold for polygraphs of strict categories.
  Indeed, considering the $2$\polygraph of strict precategories~$\P$ defined by
  \begin{align*}
    \P_0 &= \set{x}
    &
    \P_1 &= \set{f \co x \to x}
    &
    \P_2 &= \set{\alpha \co \unit x \To f}
    ,
  \end{align*}
  we have $\alpha \comp_0 \unit f \neq \unit f \comp_0 \alpha$ while
  \[
    \alpha \comp_1(\alpha \comp_0 \unit f)
    =
    \alpha \comp_0 \alpha
    =
    \alpha \comp_1 (\unit f \comp_0 \alpha)
  \]
  in the free strict $2$\category $\freecat\P$. Graphically,
  \[
    \begin{tikzcd}
      x
      \ar[rr,bend left=50,equals]
      \ar[rr,bend left=25,phantom,"\alpha\!\Downarrow"]
      \ar[r,bend left,equals]\ar[r,bend right,"f"']\ar[r,phantom,"\alpha\!\Downarrow"]
      &x\ar[r,"f"']&x
    \end{tikzcd}
    =
    \begin{tikzcd}
      x
      \ar[r,bend left,equals]
      \ar[r,phantom,"\alpha\!\Downarrow"]
      \ar[r,bend right,"f"']
      &
      x
      \ar[r,bend left,equals]
      \ar[r,phantom,"\alpha\!\Downarrow"]
      \ar[r,bend right,"f"']
      &
      x
    \end{tikzcd}
    =
    \begin{tikzcd}
      x
      \ar[rr,bend left=50,equals]
      \ar[rr,bend left=25,phantom,"\alpha\!\Downarrow"]
      \ar[r,"f"']
      &
      x
      \ar[r,bend left,equals]\ar[r,bend right,"f"']\ar[r,phantom,"\alpha\!\Downarrow"]
      &
      x
    \end{tikzcd}
  \]
\end{rem}

\paragraph*{Conduché functors.}
We now introduce the notion of (strict) Conduché functor for precategories,
following the work of Guetta in the case of strict
categories~\cite{guetta2020polygraphs}. Informally, these functors have a
\eq{co-functoriality} property, in the sense that cells mapped to composites are
themselves composites. The notion of weak Conduché functor was introduced by
Guiraud in a seemingly unrelated context~\cite{giraud1964methode} as a necessary
and sufficient condition for a functor $F \co C \to D$ between strict
$n$\categories to be \emph{exponentiable}, \ie for the pullback functor
$F^\leftarrow \co \Cat/D \to \Cat/C$ to have a right adjoint.

Let~$n \in \Ninf$,~$C,D \in \nPCat n$ and~$F \co C \to D$ be an
$n$\prefunctor. We say that~$F$ is \index{Conduché functor}\emph{$n$\Conduche}
when it satisfies that, for all~$i,k_1,k_2,k\in \N^*_n$
with~$i = \min(k_1,k_2) - 1$ and~$k = \max(k_1,k_2)$,~$u \in C_k$,
$i$\composable~$v_1 \in D_{k_1}$ and~$v_2 \in D_{k_2}$ such
that
\[
  F(u) = v_1 \pcomp_i v_2,
\]
there exist unique $i$\composable~$u_1 \in C_{k_1} $ and~$u_2\in C_{k_2}$ such
that
\[
  F(u_1) = v_1 \qqtand F(u_2) = v_2 \qqtand u_1 \pcomp_i u_2 = u.
\]

As in the case of strict categories, the Conduché property implies a unique
lifting of identities:
\begin{prop}
  \label{prop:conduche-units}
  Given $n \in \Ninf$ and an $n$\Conduche prefunctor~$F\co C \to D \in \nPCat
  n$, if
  \[
    F(u) = \unit v
  \]
  for some~$k \in \Nlt{n}$,~${u \in C_{k+1}}$, and~${v \in D_k}$, then there
  exists a unique~$u' \in C_k$ such that
  \[
    F(u') = v
    \qqtand
    u = \unit {u'}
    \zbox.
  \]
\end{prop}
\begin{proof}
  Since $\unit v = \unit v \pcomp_k \unit v$, by the Conduché property, there
  exist unique $u_1,u_2 \in C_{k+1}$ such that $F(u_1) = \unit v$,
  $F(u_2) = \unit v$ and $u = u_1 \pcomp_k u_2$. Moreover, we have that
  $F(\unit {\csrc_k (u)}) = v$ and~$u = \unit {\csrc_k (u)} \pcomp_k u$ so that
  $u_1 = \unit {\csrc_k(u)}$ and~$u_2 = u$. Symmetrically, we have that
  $u_1 = u$ and~$u_2 = \unit{\ctgt_k(u)}$. Thus,
  $u = \unit {\csrc_k(u)} \pcomp_k \unit {\ctgt_k(u)}$, so that
  $\csrc_k(u) = \ctgt_k(u)$ and $u = \unit{u'}$ with~$u' =
  \csrc_k(u)$. Uniqueness is immediate.
\end{proof}

\noindent
Unlike for strict categories, we have the remarkable property that all free
functors of precategories are Conduché:

\begin{prop}
  \label{prop:conduche}
  Given $m \in \Ninf$ and a morphism $F \co \P \to \Q \in\nPol m$, the
  prefunctor $\freecat F\co\freecat\P\to\freecat\Q$ is Conduché.
\end{prop}
\begin{proof}
  For the sake of simplicity, we only handle the case $m = \omega$. Suppose
  given $n \in \N$ and an $n$-cell $u\in\freecat\P_n$ such that $F(u)=\bar
  u_1\comp_{i}\bar u_2$. We reason by case analysis on the relative dimensions
  of $\bar u_1$ and $\bar u_2$.
  \begin{itemize}
  \item If $\bar u_1,\bar u_2\in\freecat\Q_n$ then $i=n-1$. By the unicity of
    normal forms and its compatibility with $\pcomp_{n-1}$, there are unique
    $u_1,u_2$ such that $\freecat F(u_k) = \bar u_k$ for $k \in \set{1,2}$ and
    $u = u_1 \pcomp_{n-1} u_2$.
  \item Suppose $\bar u_1\in \freecat \Q_{i+1}$ and
    $\bar u_2 \in \freecat\Q_{n}$.  If there are $u_1,u_2$ such that
    $\freecat F(u_k) = \bar u_k$ for $k \in \set{1,2}$ and
    $u = u_1 \comp_i u_u$, then they are unique since, by the previous point, $u_1$
    and $\csrc_{i+1}(u_2)$ are uniquely determined by
    $\csrc_{i+1}(u) = u_1 \pcomp_{i} u^-_2$ and $\freecat F(u_1) = \bar u_1$ and
    $\freecat F(u^-_2) = \csrc_{i+1}(u_2)$. Moreover, since
    $u = u_1 \pcomp_i u_2$, we have that $u_2$ is unique by
    \Cref{lem:pcomp-inj}. So unicity holds. For existence, we reason by
    induction on $n$ and $\bar u_2$.
    \begin{itemize}
    \item If $\bar u_2 = \bar E[\bar \alpha]$ for some $(i{+}1)$\context
      $\bar E = (\bar l,\bar E',\bar r)$, then, by unicity of normal forms,
      $u = E[\alpha]$ for some $\alpha \in \P$, and $(i{+}1)$\context
      $E = (l, E',r)$, and we moreover have
      $\freecat F(l) = \bar u_1 \pcomp_i \bar l$,
      $\freecat F(E[\alpha]) = \bar E[\bar \alpha]$ and
      $\freecat F(r) = \bar r$. By the first part, there are $u_1$ and
      $\tilde l$ such that $l = u_1 \pcomp_i \tilde l$, so that
      $\tilde E = (\tilde l,E',r)$ satisfies that
      $u = u_1 \pcomp_{i} \tilde E[\alpha]$, $\freecat F(u_1) = \bar u_1$ and
      $\freecat F(\tilde E[\alpha]) = \bar E[\bar \alpha]$.
    \item If $\bar u_2 = \bar E[\bar \alpha]$ for some $(j{+}1)$\context
      $\bar E = (\bar l,\bar E',\bar r)$ with $j > i$, then, by unicity of
      normal forms, $u = E[\alpha]$ for some $\alpha \in \P$, and
      $(j{+}1)$\context $E = (l, E',r)$, and we moreover have
      $\freecat F(l) = \bar u_1 \pcomp_i \bar l$,
      $\freecat F(E'[\alpha]) = \bar u_1 \pcomp_i \bar E'[\bar \alpha]$ and
      $\freecat F(r) = \bar u_1 \pcomp_i \bar r$. By the other induction
      hypothesis, there are $u^1_1,u^2_1,u^3_1$, $\tilde l, \tilde u', \tilde r$
      such that $l = u^1_1 \pcomp_i \tilde l$,
      $E'[\alpha] = u^2_1 \pcomp_i \tilde u'$ and $r = u^3_1 \pcomp_i \tilde
      r$. Since
      $u^1_1 \pcomp_i \ctgt_j(\tilde l) = \ctgt_{j}(l) = \csrc_j(\tilde u') =
      u^2_1 \pcomp_i \csrc_j(\bar E'[\bar \alpha])$ and
      $\freecat F(u^1_1) = \freecat F(u^2_1) = \bar u_1$ and
      $\freecat F(\ctgt_j(\tilde l)) = \freecat F(\csrc_j(\tilde u')) =
      \ctgt_j(\bar l)$, by unicity, we have $u^1_1 = u^2_1$ and
      $\ctgt_j(\tilde l) = \csrc_j(\tilde u')$.  Similarly, $u^2_1 = u^3_1$ and
      $\ctgt_j(\tilde u') = \csrc_j(\tilde r)$. Thus, writing $u_1$ for $u^1_1$
      and $u_2$ for $\tilde l \pcomp_j \tilde u' \pcomp_j \tilde r$, we have
      that $u = u_1 \pcomp_i u_2$ is the wanted decomposition for $u$.
    \item If
      $\bar u_2 = \bar E_1[\bar \alpha_1] \pcomp_{n-1} \cdots \pcomp_{n-1} \bar
      E_k[\bar \alpha_k]$, then, by unicity of normal forms, we have that
      $u = E_1[\alpha_1] \pcomp_{n-1} \cdots \pcomp_{n-1} E_k[\alpha_k]$ such
      that
      $\freecat F (E_l[\alpha_l]) = \bar u_1 \pcomp_i \bar E_l[\bar
      \alpha_l]$. By induction hypothesis, we get $u^l_1$ and $u^l_2$ such that
      $E_l[\alpha_l] = u^l_1 \pcomp_i u^l_2$, $\freecat F(u^l_1) = \bar u_1$ and
      $\freecat F(u^l_2) = \bar E_l[\bar \alpha_l]$. Using the same argument as
      earlier, we get that $u^1_1 = \cdots = u^k_1$ and $u^1_2,\ldots,u^k_2$ are
      $(n{-}1)$\composable so that, writing $u_1$ for $u^1_1$ and $u_2$ for
      $u^1_2 \pcomp_{n-1} \cdots \pcomp_{n-1} u^k_2$, we have a decomposition
      $u = u_1 \pcomp_i u_2$ satisfying the wanted properties.
    \end{itemize}
  \item Suppose $\bar u_1\in \freecat \Q_{n}$ and
    $\bar u_2 \in \freecat\Q_{i+1}$. This case is similar to the previous
    one.\qedhere
  \end{itemize}
\end{proof}

\begin{rem}
  As a counter-example for the above property in the context of strict
  categories, consider the polygraphs $\P$ and $\Q$ defined by
  \begin{align*}
    \P_0 &= \set{x}
           &
    \P_1 &= \emptyset
           &
    \P_2 &= \set{\alpha \co \unit x \To \unit x, \beta\co \unit x \To \unit x}
    \\
    \Q_0 &= \set{y}
           &
    \Q_1 &= \emptyset
           &
    \Q_2 &= \set{\gamma \co \unit y \To \unit y}
           \zbox.
  \end{align*}
  We then have a morphism $F \co \P \to \Q$ sending $\alpha$ and $\beta$ to
  $\gamma$, and the associated prefunctor $\freecat F$ sends both $\alpha \comp_0 \beta$
  and $\beta \comp_0 \alpha$ to $\gamma \comp_0 \gamma$.
\end{rem}

A nice application of the above Conduché properties is the characterization of
monomorphisms of polygraphs. First, we briefly observe the equivalence between
monomorphisms of precategories and dimensionwise injections.
\begin{prop}
  \label{prop:charact-pcat-mono}
  Given $n \in \Ninf$ and $F \co C \to D \in \nPCat n$, the following are equivalent:
  \begin{enumerate}[(i)]
  \item \label{prop:charact-ppol-mono:mono} $F$ is a monomorphism,
    
  \item \label{prop:charact-ppol-mono:comp-mono} $F_k$ is a monomorphism for every $k\le n$.
  \end{enumerate}
\end{prop}
\begin{proof}
  The theory of $n$\precategories is sketchable and the functor
  $(-)_k \co \nPCat n \to \Set$, which to a precategory associates its set of
  $k$-cells is induced by a sketch morphism. It is thus a right
  adjoint \cite[Section~4, Theorem~4.1]{barr2000toposes}. In particular, it preserves monomorphisms. Thus,
  \ref{prop:charact-ppol-mono:mono} implies
  \ref{prop:charact-ppol-mono:comp-mono}. Moreover, since the functors $(-)_k$
  for $k < n+1$ are jointly faithful, we have that
  \ref{prop:charact-ppol-mono:comp-mono} implies
  \ref{prop:charact-ppol-mono:mono}.
\end{proof}
\noindent We then have the following characterization property for monomorphisms
of polygraphs, which are in particular preserved by the functor $\freecat-$:
\begin{prop}
  \label{prop:charact-ppol-mono}
  Given $n \in \Ninf$ and $F \co \P \to \Q \in \nPol n$, the following are equivalent:
  \begin{enumerate}[(i)]
  \item \label{prop:charact-ppol-mono:mono} $F$ is a monomorphism,
    
  \item \label{prop:charact-ppol-mono:comp-mono} $F_k$ is a monomorphism for every $k\le n$,
  \item \label{prop:charact-ppol-mono:freefunct-mono} $\freecat F$ is a
    monorphism in $\nPCat n$.
  \end{enumerate}
\end{prop}
\begin{proof}
  We show this property by induction on $n$.
  \Cref{prop:charact-ppol-mono:comp-mono} clearly implies
  \Cref{prop:charact-ppol-mono:mono}.

  Conversely, assuming \Cref{prop:charact-ppol-mono:mono}, by induction
  hypothesis, we have that $F_k$ and $\freecat F_k$ are monomorphisms for $k <
  n$. Now, let $x,y \in \P_n$ such that $F_n(x) = F_n(y)$. In particular, we
  have $\freecat F_{n-1}(\csrctgt\eps(x)) = \freecat F_{n-1}(\csrctgt\eps(y))$
  for $\eps \in\set{-,+}$, so that $\csrctgt\eps(x) = \csrctgt\eps(y)$ for
  $\eps\in\set{-,+}$, by injectivity of $\freecat F_{n-1}$. Consider the
  $n$\polygraph $\R$ such that $\poltruncf_{n-1} \R = \poltruncf_{n-1} \P$ and
  $\R_n = {z}$ with $\gsrctgt\eps_{n-1}(z) =\csrctgt\eps(x)$ for $\eps \in
  \set{-,+}$. Then, we have two canonical morphisms $G^x,G^y \co \R \to \P$,
  verifying $G^x(z) = x$ and $G^y(z) = y$. We then have $F\circ G^x = F\circ
  G^y$, so that $G^x = G^y$ since $F$ is a monomorphism. In particular, we have
  $x = y$. Thus, $F_n$ is injective, so \Cref{prop:charact-ppol-mono:comp-mono}
  holds.

  By \Cref{thm:canonical-form}, the embedding
  $\polinje^\P_k:\P_k\to\freecat \P_k$ (\resp
  $\polinje^\Q_k:\Q_k\to\freecat \Q_k$) is a monomorphism. Thus,
  \Cref{prop:charact-ppol-mono:freefunct-mono} implies
  \Cref{prop:charact-ppol-mono:comp-mono}, since
  $\polinje^\Q_{k} \circ F_k = \freecat F_k \circ \polinje^\P_{k}$ and the
  right-hand side of the latter equation is a monomorphism by
  \Cref{prop:charact-pcat-mono}.

  Conversely, assume \Cref{prop:charact-ppol-mono:comp-mono}. Let $u,v \in
  \freecat\P_n$ such that $\freecat F(u) = \freecat F(v)$. We show that $u = v$
  by induction on an expression defining $u$. If $u = \unit{u'}$ for some $u'
  \in \freecat\P_{n-1}$, by \Cref{prop:conduche,prop:conduche-units}, there
  exists $v' \in \freecat\P_{n-1}$ such that $v = \unit{v'}$. We thus have
  $\freecat F(u') = \freecat F(v')$ and $u' = v'$ by induction hypothesis. If $u
  = u_1 \comp_i u_2$ for some $i < n$ and $i$\composable $u_1,u_2 \in
  \freecat\P$, then by \Cref{prop:conduche}, there exists $i$\composable
  $v_1,v_2 \in \freecat\P$ such that $v = v_1\comp_i v_2$ and $\freecat F(u_k) =
  \freecat F(v_k) $ for $k \in \set{1,2}$, so that $u_k = v_k$ by induction
  hypothesis, and $u = v$. Finally, if $u = \polinje_{\P,n}(g)$ for some $g \in \P_n$
  then $v = \polinje^\P_n(h)$ for some $h \in \P_n$ by \Cref{prop:freecat-functors}.
  But then, we have
  \[
    \polinje^\Q_n(F_n(g))
    =
    \freecat F(\polinje^\P_n(g))
    =
    \freecat F(\polinje^\P_n(h))
    =
    \polinje^\Q_n(F_n(h))
  \]
  where $\polinje^\Q_n \circ F_n$ is a monomorphism by hypothesis and
  \Cref{rem:polinj-injective}. Thus, $g = h$ and $u = v$. Hence,
  \Cref{prop:charact-ppol-mono:freefunct-mono} holds.
\end{proof}


\section{Makkai's criterion for presheaf categories}
\label{sec:makkai}
We now recall the criterion given by Makkai~\cite{makkai2005word} to detect
whether a category $\cC$ is a presheaf category in the expected way, \ie
relatively to a concretization functor $\cC \to \Set$. In the case of a presheaf
category, the objects of the base category are recognized as the \eq{suitably
  initial} elements of the concretization. Makkai used this criterion to show
that polygraphs for strict categories do not form a presheaf categories in the
expected way, where the concretization functor maps a polygraph to the set of
all generators. We will use this criterion in \Cref{sec:presheaf} to prove that,
in the case of precategories, we do get a presheaf category.

\medskip

A \index{concrete category}\emph{concrete category} is a category~$\mcal C$
endowed with a functor
\[
  \concrf[\mcal C]- \co \mcal C \to \Set\zbox.
\]
The above \index{concretization functor}\emph{concretization functor} should be
understood as a candidate set-theoretic representation of~$\mcal C$: for
$c$ an object of $C$, the set $\concrf c$ describes the candidate elements of the associated
presheaf. The following canonical example should provide a good illustration of
this intuition.

\begin{example}
  \label{ex:can-concrete-preasheaf-cat}
  Let~$C$ be a small category.~$\hat C$ has a canonical structure of concrete
  category, where~$\concrf[\hat C]-$ is defined on preasheaves~$P \in \hat C$ by
  \[
    \concrf[\hat C]{P} = \bigsqcup_{c \in C_0} P(c)
  \]
  and extended naturally to morphisms between presheaves.
\end{example}
\noindent In the following, we will be interested in the concretization functor
given by the following example:
\begin{example}
  \label{ex:pol-concrete}
  The functor~$\poltoset{-}\co \nPol\omega \to \Set$ which maps~$\P \in \nPol\omega$ to
  \[
    \poltoset \P = \bigsqcup_{k \in \N} \P_k
  \]
  equips~$\nPol\omega$ with a structure of concrete category.
\end{example}
Later, we will study the properties of~$\oPol$ equipped with the above
concretization functor. Another concretization functor on~$\oPol$ that will be
of interest for us is given by the example below:
\begin{example}
  \label{ex:polconcrb}
  There is a functor~$\cattoset-\co \nCat\omega \to \Set$ which
  maps~$C\in\nCat\omega$ to
  \[
    \cattoset C=\bigsqcup_{k \in \N}C_k
  \]
  By precomposition with the functor $(-)^*:\nPol\omega\to\nCat\omega$, we
  obtain a functor~$\polconcrb-\co \nPol\omega \to \Set$ which
  maps~$\P \in \nPol\omega$ to
  \[
    \polconcrb\P = \bigsqcup_{k \in \N} \freecat\P_k
  \]
  and also equips~$\nPol\omega$ with a structure of concrete category.
\end{example}
\noindent
In order to distinguish between the two preceding structures of concrete
category on~$\oPol$, we use the convention that we write~$\oPol$ when
considering the concrete category structure on~$\oPol$ given by~$\poltoset-$
\glossary(Polostar){$\oPolb$}{the category~$\oPol$ equipped with the
  concretization functor~$\polconcrb -$}and~$\oPolb$ when considering the
concrete category structure on~$\oPol$ given by~$\polconcrb -$.

\medskip

An \index{equivalence of concrete categories}\emph{equivalence of concrete
  categories} between concrete categories~$(\mcal C,\concrf[\mcal C]{-})$
and~$(\mcal D,\concrf[\mcal D]{-})$ is the data of an equivalence of
categories~$\mcal E\co \mcal C \to \mcal D$ and a natural isomorphism
\[
  {\Phi\co \concrf[\mcal D]{-} \circ \mcal E \To \concrf[\mcal C]{-}}\zbox.
\]
When such an equivalence exists,~$(\mcal C,\concrf[\mcal C]{-})$ and~$(\mcal
D,\concrf[\mcal D]{-})$ are said \index{concretely equivalent
  categories}\emph{concretely equivalent}. One might then consider the following
natural question:
\begin{center}
  When is some concrete category~$(\mcal C,\concrf[\mcal C]{-})$ concretely equivalent\\
  to a presheaf category~$(\hat C,\concrf[\hat C]{-})$ for some small
  category~$C$?
\end{center}
When it is the case, we say that~$(\mcal C,\concrf[\mcal C]{-})$ is a
\index{concrete presheaf category}\emph{concrete presheaf category}.

\medskip

Given a concrete category~$(\mcal C,\concrf[\mcal C]{-})$, the
\emph{category of elements}~$\Elt(\mcal C)$ of~$\mcal C$ is the category
\begin{itemize}
\item whose objects are the pairs~$(X,x)$ where~${X \in \mcal C_0}$ and~${x \in
    \concrf[\mcal C]{X}}$, and
\item whose morphisms from~$(X,x)$ to~$(Y,y)$ are the morphisms~$f \co X \to
  Y \in \mcal C$ such that~$\concrf[\mcal C]{f}(x) = y$.
\end{itemize}
Given a morphism~$f\co (X,x) \to (Y,y)$ as above, we say that~$y$ is a
\index{specialization}\emph{specialization} of~$x$.
An object~$(X,x) \in \Elt(\mcal C)$ is \index{principal!element}\emph{principal} when, for every
morphism~$f\co (Y,y) \to (X,x) \in \Elt(\mcal C)$ such that~$f$ is a
monomorphism in~$\mcal C$, we have that~$f$ is an isomorphism; it is
\index{primitive element}\emph{primitive} when it is principal and, for all~$f\co (Y,y) \to (X,x)
\in\Elt(\mcal C)$ where~$(Y,y)$ is principal,~$f$ is an isomorphism.
\begin{example}
  Let~$C$ be a small category and consider the canonical concrete category
  structure on~$\hat C$ given by \Cref{ex:can-concrete-preasheaf-cat}. Given
  $P\in\hat C$ and $c\in C$, we write $\iota_c:P(c)\to\bigsqcup_{c\in C}P(c)$
  for the canonical injection. The category~$\Elt(\hat C)$ has
  \begin{itemize}
  \item as objects the pairs~$(P,\iota_c(x))$ where~$P \in \hat C$ and~$x \in
    P(c)$, and
  \item as morphisms from~$(P,\iota_c(x))$ to~$(Q,\iota_d(y))$ the natural
    transformations~$\alpha\co P \To Q$ such that~$c = d$ and~$\alpha_c(x) = y$.
  \end{itemize}
  \noindent Given~$(P,\iota_c(x)) \in \Elt(\hat C)$, we have the following.
  \begin{itemize}
  \item $(P,\iota_c(x))$ is principal when~$P$ is the smallest subpresheaf~$P'$
    of~$P$ such that~$x \in P'(c)$. In particular, for
    all~$c \in C$,~$(C(-,c),\iota_c(\unit c))\in \Elt(\hat C)$ is principal.
  \item $(P,\iota_c(x))$ is primitive when the natural transformation~$\theta\co
    C(-,c) \to P$ which maps~$\unit c$ to~$x$ is an isomorphism.
  \end{itemize}
\end{example}
\noindent The characterization of concrete presheaf categories given by Makkai
is the following~\cite[Theorem~4]{makkai2005word}:
\begin{theo}
  \label{thm:concrete-charact}
  Let~$(\mcal C,\concrf[\mcal C]{-})$ be a concrete category. $\mcal C$ is
  concretely equivalent to a presheaf category if and only if the following
  conditions are all satisfied:
  \begin{enumerate}[(a)]
  \item \label{thm:concrete-charact:iso} $\concrf[\mcal C]{-}$ reflects isomorphisms,
  \item \label{thm:concrete-charact:colimits} $\mcal C$ is cocomplete
    and~$\concrf[\mcal C]{-}$ preserves all small colimits,
    
  \item \label{thm:concrete-charact:small} the collection of isomorphism
    classes of primitive elements of~$\Elt(\mcal C)$ is small,
    
  \item \label{thm:concrete-charact:prim-exists} for every element~$(X,x) \in
    \Elt(C)$, there is a morphism~$(U,u) \to (X,x)$ for some primitive
    element~$(U,u)$,
    
  \item \label{thm:concrete-charact:prim-equal} given two morphisms~$f,g\co
    (U,u) \to (X,x) \in \Elt(\mcal C)$ where~$(U,u)$ is primitive, we have~$f =
    g$,
    
  \item \label{thm:concrete-charact:prim-iso} given two morphisms~$f\co (U,u)
    \to (X,x)$ and~$g\co (V,v) \to (X,x)$ of~$\Elt(\mcal C)$ where both~$(U,u)$
    and~$(V,v)$ are primitive, there is an isomorphism~$\theta\co (U,u) \to
    (V,v)$ such that~$g \circ \theta = f$. 
  \end{enumerate}
\end{theo}

\section{The support function}
\label{sec:support}
It is often useful to consider the support of a cell in a precategory, which
informally consists in the set of generators occurring in this cell. In
particular, the support will allow us to retrieve some properties of a morphism
of polygraphs $F$ from the associated free functor $\freecat F$, which will turn
out to be useful when studying polyplexes. A support function for free strict
categories was already introduced by Makkai for his study of the word problem on
these categories~\cite{makkai2005word}.

\medskip

Given $n \in \Ninf$ and an $n$\polygraph~$\P$, we define the
\index{support function}\emph{support function}
\[
  \supp[\P]\co \cattoset{\freecat\P}
  \to
  \pset{\poltoset\P}
\]
which to any cell in $\freecat\P$ associates a set of generators of~$\P$, by
induction on $u \in \freecat\P$ as follows:
\begin{itemize}
\item if $u = g \in \P_0$, then $\supp(u) = \set g$,
\item if $u = g \in \P_{k+1}$ for some $k < n$, then $\supp(u) = \set g \cup
  \supp(\gsrc(g)) \cup \supp(\gtgt(g))$,
\item if $u = \unit{u'}$ for some $k < n$ and $u' \in \freecat{\P}_{k}$, then $\supp(u) = \supp(u')$,
\item if $u = u_1 \pcomp_i u_2$ for some $0 < k_1,k_2 < n+1$, $i =
  \min(k_1,k_2)-1$ and $i$\composable $u_1 \in \freecat{\P_{k_1}}$ and $u_2 \in
  \freecat{\P_{k_2}}$, then $\supp(u) = \supp(u_1) \cup \supp(u_2)$.
\end{itemize}

One can easily verify that $\supp$ respects the axioms of precategories, so that:
\begin{lem}
  The function $\supp$ is well-defined.
\end{lem}
\noindent
The function $\supp$ is moreover natural:
\begin{lem}
  \label{lem:supp-natural}
  Let $n \in \Ninf$ and~$F \co \P \to \Q \in \nPol n$. Then, we have that
  $\supp[\Q] \circ \cattoset{\freecat F} = \poltoset F \circ \supp[\P]$.
\end{lem}
\begin{proof}
  By induction on $u \in \freecat\P$.
\end{proof}

\noindent
Given a polygraph~$\P$ and a cell $u \in \freecat\P$, the support of $u$ is
always finite. By restricting~$\P$ to the generators occurring in this support,
on can show the following:
\begin{prop}
  \label{prop:pol-lifting-cell-ex}
  Given~$n \in \Ninf$, an $n$\polygraph~$\P$ and~$u \in \freecat\P$, there
  exist a finite $n$\polygraph~$\tilde{P}$, a monomorphism~$F\co \tilde{P}
  \to \P$ and~$\tilde u \in \freecat{{\tilde{P}}}$ such that~$\freecat F(\tilde
  u) = u$ and~$\supp(\tilde u) = \poltoset{\tilde\P}$.
\end{prop}

Given $F \co \P \to \Q$ and $u \in \freecat\P$, we write $F/u \co \supp(u) \to
\supp(\freecat F (u))$ for the restriction of $F$ to the support and the image
of the support of $u$.
\begin{lem}
  \label{lem:oppol-morph-restr}
  Given a pair of parallel morphisms
  \[
    \begin{tikzcd}
      \P
      \ar[r,shift left,"F"]
      \ar[r,shift right,"G"']
      &
      \Q
    \end{tikzcd}
  \]
  of $\oPol$ such that $\freecat F(u) = \freecat G(u)$ for some $u \in
  \freecat \P$, we have $F/u = G/u$.
\end{lem}
\begin{proof}
  By induction on $n$ and a formula defining $u$.
  \begin{itemize}
  \item If $u = \alpha$ for some $\alpha \in \P$, then $F(\alpha) =
    G(\alpha)$. We then also have that
    $\freecat F(\csrctgt\eps(\alpha)) =\freecat G(\csrctgt\eps(\alpha))$ for
    $\eps \in \set{-,+}$, so that $F/\csrctgt\eps(u) =G/\csrctgt\eps(u)$ by
    induction.  Thus, $F/\alpha = G/\alpha$.
  \item If $u = \unit{u'}$, then the property follows by induction hypothesis.
  \item If $u = u_1 \pcomp_i u_2$. Then, we have $\freecat F(u_1) \pcomp_i
    \freecat F(u_2) = \freecat G(u_1) \pcomp_i \freecat G(u_2)$. Writing
    $\toterm[\P]\co \P \to \polterm$ for the terminal morphism in $\oPol$,
    we have $\freecat{\toterm[\Q]}(\freecat
    F(u_j))=\freecat{\toterm[\P]}(u_j)=\freecat{\toterm[\Q]}(\freecat G(u_j))$
    for $j \in \set{1,2}$. Since $\freecat{\toterm[\P]}$ is Conduché by
    \Cref{prop:conduche}, we have $\freecat F(u_j) = \freecat G(u_j)$ for $j \in
    \set{1,2}$. Thus, $F/u_j = G/u_j$ for $j \in \set{1,2}$ so that $F / u = G / u$.\qedhere
  \end{itemize}
\end{proof}

\noindent We have the following nice description of principal elements of
$\Elt(\oPol)$ and $\Elt(\oPolb)$ using support:
\begin{lem}
  \label{lem:principal-charact}
  An element $(\P,u)$ of $\Elt(\oPol)$ (\resp $\Elt(\oPolb)$) is principal if
  and only if $\supp(u) = \poltoset{\P}$.
\end{lem}
\begin{proof}
  Assume that $(\P,u)$ is principal. Then, by \Cref{prop:pol-lifting-cell-ex},
  there exist an element~$(\tilde\P,\tilde u)$ and $F \co(\tilde\P,\tilde u) \to
  (\P,u)$ such that $F$ is a monomorphism of $\oPol$ and $\supp(\tilde u) =
  \poltoset{\tilde\P}$. Since $(\P,u)$ is principal, we have that $F$ is an
  isomorphism. Thus, by \Cref{lem:supp-natural}, we have that $\supp(u) =
  \poltoset\P$.

  Conversely, assume that $\supp(u) = \poltoset\P$. Let $(\Q,v)$ be an element
  and $F \co (\Q,v) \to (\P,u)$ be a morphism where $F$ is a monomorphism in
  $\oPol$. By \Cref{lem:supp-natural}, we have that $\poltoset{F(\supp(v))} =
  \supp(u) = \poltoset\P$. Thus, $F_k$ is surjective for every~$k \in \N$.
  Moreover, $F_k$ is injective by \Cref{prop:charact-ppol-mono}, so that $F_k$
  is an isomorphism for every~$k$. Since $\poltoset -$ reflects isomorphisms
  (exercise to the reader), we have that $F$ is an isomorphism. Thus, $(\P,u)$
  is principal.
\end{proof}
\noindent
Finally, as a consequence of \Cref{lem:oppol-morph-restr,lem:principal-charact},
we have:

\begin{lem}
  \label{lem:principal-parallel-eq}
  Given a pair of parallel morphisms
  \[
    \begin{tikzcd}
      (\P,u)
      \ar[r,shift left,"F"]
      \ar[r,shift right,"G"']
      &
      (\Q,v)
    \end{tikzcd}
  \]
  of $\Elt(\oPolb)$ where $(\P,u)$ is principal, then $F = G$.
\end{lem}


\section{Polyplexes}
\label{sec:polyplex}
We now introduce the construction of polyplexes for the cells of free
precategories. Those are polygraphs representing composition shapes such that
every such cell in a polygraph is the composite of a polyplex in a unique way.
Polyplexes are themselves composed of plexes (see next section) which are
polygraphs representing generators in a polygraph.
%
%
These notions are due to Burroni~\cite{burroni2012automates}, and were further
developed by Henry~\cite{henry2018nonunital}.

Formally, a \emph{polyplex} is an element $(\P,u)\in\Elt(\oPolb)$ which is
primitive (for the concrete structure introduced in \Cref{ex:polconcrb}). Given
an element $(\Q,v)$ in $\Elt(\oPolb)$, a \emph{polyplex lifting} is the data of
a polyplex $(\P,u)$ and a morphism of elements $F \co (\P,u) \to (\Q,v) \in
\Elt(\oPolb)$.

The construction of polyplexes will be carried out by induction on a formula
defining a cell. The inductive case of identities is handled by the following
lemma:
\begin{lem}
  \label{lem:polyplex-unit}
  Given an element $(\P,u) \in \Elt(\oPolb)$, $(\P,u)$ is a polyplex if and
  only if $(\P,\unit u)$ is a polyplex.
\end{lem}
\begin{proof}
  By \Cref{lem:principal-charact}, $(\P,u)$ is principal if and only if
  $(\P,\unit u)$ is principal. So we can assume that both are principal.

  Suppose that $(\P,u)$ is primitive. Let $F \co (\Q,v) \to (\P,\unit u)$ be a
  morphism of elements where $\Q$ is principal. Then, by
  \Cref{prop:freecat-functors}, we have that $v = \unit {v'}$ for some $v' \in
  \freecat\Q$, and, by compatibility of $\freecat F$ with $\csrc$, we moreover
  have $\freecat F(v') = u$. Since $\supp(v) = \supp(v')$, $(\Q,v')$ is still a
  principal element. Thus, $F$ is an isomorphism since $(\P,u)$ is primitive.
  Hence, $(\P,\unit u)$ is primitive. The converse is similar.
\end{proof}

The lemmas and propositions until the end of this section, describing the
remaining cases characterizing polyplexes for composites and generators together
with global existence and unicity properties, are proved by \underline{mutual
  induction} on a formula defining the cell $u$ appearing in the
statements. First, the case of generators:

\begin{lem}
  \label{lem:polyplex-gen}
  Let $(U,u) \in \Elt(\oPolb)$. Then, the following are equivalent:
  \begin{enumerate}[label=(\roman*),ref=(\roman*)]
  \item \label{lem:polyplex-gen:gen} $(U,u)$ is a polyplex and there exist
    $\alpha \in U$ such that $u = \alpha$,
  \item \label{lem:polyplex-gen:pushout} there exist polyplex
    liftings
    \[
      G^\eps \co (U^\eps,u^\eps) \to (U,\csrctgt\eps(u))
    \]
    for $\eps \in \set{-,+}$, principal elements $(S,s)$ and $(T,t)$, and
    morphisms
    \begin{align*}
      F^{\eps-} \co (S,s)&\to (U^\eps,\csrc(u^\eps))
      &
      F^{\eps+} \co (T,t)&\to (U^\eps,\ctgt(u^\eps))      
    \end{align*}
    for $\eps \in \set{-,+}$, such that, considering the pushout
    \[
      \begin{tikzcd}[column sep=20mm]
        S \sqcup T
        \ar[r,"{[F^{+-},F^{++}]}"]
        \ar[d,"{[F^{--},F^{-+}]}"']
        &
        U^+
        \ar[d,"\bar G^+",dotted]
        \\
        U^-
        \ar[r,"\bar G^-"',dotted]
        &
        \partial U
      \end{tikzcd}
    \]
    $(U,u)$ is isomorphic to $(\bar U,\bar \alpha)$, where $\bar U$ is
    obtained from $\partial U$ by adding a generator
    \[
      \bar \alpha : \bar G^-(u^-) \to \bar G^+(u^+)
      \zbox.
    \]
  \end{enumerate}
\end{lem}

\begin{proof}
  Suppose that \Cref{lem:polyplex-gen:pushout} holds. By the unicity of normal
  forms (\Cref{thm:canonical-form}), it is enough to show that
  $(\bar U,\bar \alpha)$ is primitive. First, it is principal by
  \Cref{lem:principal-charact} since
  \[
    \supp(\bar \alpha) = \set{\alpha} \cup \supp(\bar G^-(u^-)) \cup \supp(\bar G^+(u^+)) = \poltoset{\bar U}
    \zbox.
  \]
  Second, consider a morphism $H \co (V,v) \to (\bar U,\alpha)$ with $(V,v)$
  principal. By induction hypothesis on~\Cref{prop:primitive-lifting}, we have
  polyplex liftings
  \[
    H^\eps \co (\tilde U^\eps,\tilde u^\eps) \to (V,\csrctgt\eps(v))
  \]
  for $\eps \in \set{-,+}$. Since
  $\freecat H(\csrctgt\eps(v)) = \csrctgt\eps(\bar u)$,
  by~\Cref{prop:polyplex-lifting-unique}, we can assume that
  $(\tilde U^\eps,\tilde u^\eps) = (U^\eps,u^\eps)$ for $\eps \in
  \set{-,+}$. Since~$(S,s)$ is principal, we have, by
  \Cref{lem:principal-parallel-eq}
  \[
    H^- \circ F^{--} = H^+ \circ F^{+-}
    \qqtand
    H^- \circ F^{-+} = H^+ \circ F^{++}
    \zbox.
  \]
  Thus, we derive a morphism $\partial H' \co \partial U \to V$ from the
  pushout. By unicity of normal forms, $v = \beta$ for some $\beta \in V$. Thus,
  $\partial H'$ can be extended to $H' \co \bar U \to V$ by putting
  $H'(\alpha) = \beta$. Using~\Cref{lem:principal-parallel-eq}, we can easily
  verify that $H'$ is the inverse of $H$. Hence, $(\bar U,\bar u)$ is a
  polyplex.

  Now, assume that \Cref{lem:polyplex-gen:gen} holds. By induction hypothesis,
  there are polyplex liftings
  \[
    G^\eps \co (U^\eps,u^\eps) \to (\P,\csrctgt\eps(u))
  \]
  for $\eps \in \set{-,+}$. By induction hypothesis
  on~\Cref{prop:primitive-lifting}, there exists a polyplex lifting
  $F^{--} \co (S,s) \to (U^-,\csrc(u^-))$. Similarly, there is a polyplex
  lifting of $(U^+,\csrc(u^+))$ and, since
  $\freecat {(G^-)}(\csrc(u^-)) = \freecat{(G^+)}(\csrc(u^+))$,
  by~\Cref{prop:polyplex-lifting-unique}, it can be chosen to be of the form
  \[
    F^{+-} \co (S,s) \to (U^+,\csrc(u^+))
    \zbox.
  \]
  Similarly, there are polyplex liftings
  \[
    F^{-+} \co (T,t) \to (U^-,\ctgt(u^-))
    \qqtand
    F^{++} \co (T,t) \to (U^+,\ctgt(u^+))
    \zbox.
  \]
  Writing $F^\eps$ for $[F^{\eps-},F^{\eps+}]$ for $\eps \in \set{-,+}$,
  consider the pushout
  \[
    \begin{tikzcd}
      S \sqcup T
      \ar[r,"F^+"]
      \ar[d,"F^+"']
      &
      U^+
      \ar[d,"\bar G^+",dotted]
      \\
      U^-
      \ar[r,"\bar G^-"',dotted]
      &
      \partial U
    \end{tikzcd}
  \]
  and write $\bar U$ for the \opol obtained from $\partial U$ by adding a generator
  $\bar \alpha \co \bar G^-(u^-) \to \bar G^+(u^+)$ (this is well-defined, since
  the definition of $\partial U$ ensures that $\csrctgt\eps(\bar G^-(u^-)) =
  \csrctgt\eps(\bar G^+(u^+))$ for $\eps \in \set{-,+}$). By the first part,
  $(\bar U,\bar \alpha)$ is a polyplex, and we easily deduce a polyplex lifting $H
  \co (\bar U,\bar \alpha) \to (U,\alpha)$ from the above pushout. Since
  $(U,\alpha)$ is primitive, $H$ is an isomorphism.
  Thus,~\Cref{lem:polyplex-gen:pushout} holds.
\end{proof}
\noindent The next lemma deals with the case of composites of the polyplex construction:
\begin{lem}
  \label{lem:polyplex-comp}
  Let $(U,u) \in \Elt(\oPolb)$, $u_1 \in \freecat U_k$, $u_2 \in \freecat U_l$
  for some $k,l \in \N$, with $u_1$ and $u_2$ are $i$\composable for $i =
  \min(k,l)$. Then, the following are equivalent:
  \begin{enumerate}[label=(\roman*),ref=(\roman*)]
  \item \label{lem:polyplex-comp:comp} $(U,u)$ is a polyplex and $u = u_1
    \pcomp_i u_2$,
  \item \label{lem:polyplex-comp:pushout} there exist a principal element
    $(U',u')$ and polyplexes $(U_1,u_1)$ and $(U_2,u_2)$, and morphisms $F_j \co
    U'\to U_j \in \oPol$ and $G_j \co U_j \to U$ for $j \in \set{1,2}$, such
    that
    \[
      \begin{tikzcd}
        U'
        \ar[d,"F_1"']
        \ar[r,"F_2"]
        &
        U_2
        \ar[d,"G_2"]
        \\
        U_1
        \ar[r,"G_1"']
        &
        U
      \end{tikzcd}
    \]
    is a pushout diagram in $\oPol$, $\freecat F_1(u') = \ctgt_i(u_1)$, $\freecat F_2(u')
    = \csrc_i(u_2)$ and $u = G_1(u_1) \pcomp_i G_2(u_2)$.
  \end{enumerate}
\end{lem}

\begin{proof}
  Suppose that \Cref{lem:polyplex-comp:pushout} holds.
  We have
  \begin{align*}
    \supp(\freecat G_1(u_1) \pcomp_i \freecat G_2(u_2))
    &= \supp(\freecat G_1(u_1)) \cup
    \supp(\freecat G_2(u_2)) \\
    &= G_1(\supp(u_1)) \cup G_2(\supp(u_2)) \\
    & = G_1(\poltoset{U_1}) \cup G_2(\poltoset{U_2})\\
    &= \poltoset{U} 
  \end{align*}
  thus $(U,u)$ is principal by \Cref{lem:principal-charact}. Now, consider $H
  \co (\R,w) \to (U,u) \in \Elt(\oPolb)$ with $(\R,w)$ principal. We have
  \[
    H(\poltoset{\R}) = H(\supp(w)) =
    \supp(\freecat H(w)) = \supp(u) = \poltoset{U}
  \]
  so that the functions $H_j \co \R_j \to U_j$ are surjective for every $j$.
  Thus, $H$ is an epimorphism. Since $\freecat H$ is Conduché
  by~\Cref{prop:conduche} and $\freecat H(w) = \freecat G_1(u_1) \pcomp_i
  \freecat G_2(u_2)$, there exist unique $w_1,w_2$ such that $\freecat H(w_j) =
  \freecat G_j(u_j)$ for $j \in \set{1,2}$ and $w = w_1 \pcomp_i w_2$. By
  induction hypothesis on \Cref{prop:primitive-lifting}, there exist polyplex
  liftings $H'_j \co (\tilde U_j,\tilde u_j) \to (\R,w_j)$ for $j \in
  \set{1,2}$. By induction hypothesis on \Cref{prop:polyplex-lifting-unique},
  since both $(\tilde U_j,\tilde u_j)$ and $(U_j,u_j)$ are polyplex liftings of
  $(U,\freecat G_j(u_j))$, we may assume that $(\tilde U_j,\tilde u_j)
  =(U_j,u_j)$ for $j \in \set{1,2}$. By \Cref{lem:principal-parallel-eq}, we
  have $H'_1 \circ F_1 = H'_2 \circ F_2$, so that we obtain $H' \co U \to \R$
  from the pushout. We compute that
  \[
    H'(u) = H'(\freecat G_1(u_1) \pcomp_i \freecat G_2(u_2)) = H'_1(u_1)
    \pcomp_i H'_2(u_2) = w_1 \pcomp_i w_2 = w\zbox.
  \]
  Thus, using \Cref{lem:principal-parallel-eq}, we easily have that $H' \circ H
  = \unit {\R}$ and $H \circ H' = \unit{U}$. Hence, $(U,u)$ is primitive.

  Conversely, suppose that \Cref{lem:polyplex-comp:comp} holds. Then, by
  induction hypothesis on \Cref{prop:primitive-lifting}, there exist $G_k \co
  (U_k,\bar u_k) \to (\P,u_k)$ with $(U_k,\bar u_k)$ primitive for $k \in
  \set{1,2}$. By induction hypothesis on \Cref{prop:primitive-lifting}, there
  exist $\tilde F_k \co (\tilde U_k,\tilde u_k) \to
  (U_k,\csrctgt{\eps(k)}_i(u_k))$ with $(\tilde U_k,\tilde u_k)$ primitive for
  $i \in \set{1,2}$ and $\eps(1) = +$ and $\eps(2) = -$. In particular, $(\tilde
  U_1,\tilde u_1)$ and $(\tilde U_2,\tilde u_2)$ are both polyplex liftings of
  $(U,\ctgt_i(u_1))$. By induction hypothesis on
  \Cref{prop:polyplex-lifting-unique}, we can assume that $(\tilde U_1,\tilde
  u_1) = (\tilde U_2,\tilde u_2)$ and write $(\tilde U,\tilde u)$ for this
  element. By \Cref{lem:principal-parallel-eq}, we have $G_1 \circ F_1 = G_2
  \circ F_2$. Consider the pushout
  \[
    \begin{tikzcd}
      \tilde U
      \ar[d,"F_1"']
      \ar[r,"F_2"]
      &
      U_2
      \ar[d,"\bar G_2",dotted]
      \\
      U_1
      \ar[r,"\bar G_1"',dotted]
      &
      \bar U
    \end{tikzcd}
  \]
  By its universal property, we get a morphism $H \co \bar U \to U$ from $G_1$
  and $G_2$. By the first implication, $(\bar U,\freecat G_1(\bar u_1) \pcomp_i
  \freecat G_2(\bar u_2))$ is a primitive element. Moreover, $H$ induces a
  morphism
  \[
    H \co (\bar U,\freecat G_1(\bar u_1) \comp_i \freecat G_2(\bar u_2)) \to (U,u_1\pcomp_i u_2)
  \]
  of $\Elt(\oPolb)$. Thus, since $(U,u_1\pcomp_i u_2)$ is primitive, $H$
  is an isomorphism.
\end{proof}

\noindent
The previous lemmas lead to the following polyplex lifting existence property:

\begin{prop}
  \label{prop:primitive-lifting}
  Given an element $(\P,u)\in\Elt(\oPolb)$, there exists a polyplex lifting
  \[
    F \co (U,\bar u) \to (\P,u)
  \]
  where~$(U,\bar u)$ is primitive.
\end{prop}
\begin{proof}
  We reason by case analysis on a formula for~$u$.
  \begin{itemize}
  \item If $u = \unit {u'}$, then, by \Cref{lem:polyplex-unit}, the conclusion
    follows from induction hypothesis.
  \item If $u = u_1 \pcomp_i u_2$, then, by induction hypothesis, there are
    morphisms
    \[
      G^k \co (U^k,\bar u_k) \to (\P,u)
    \]
    with $(U^k,\bar u_k)$ primitive for $k \in \set{1,2}$. By induction
    hypothesis, there are polyplex liftings
    \[
      F^k \co (\tilde U^k,\tilde u_k) \to (U^k,\csrctgt{\eps(k)}_i(\bar u_k))
    \]
    with $\eps(1) = +$ and $\eps(2) = -$. By induction hypothesis on
    \Cref{prop:polyplex-lifting-unique}, we can assume that
    $(\tilde U^1,\tilde u_1) = (\tilde U^2,\tilde u_2)$ and write
    $(\tilde U,\tilde u)$ for this element. Since $(\tilde U,\tilde u)$ is
    principal, we have $G^1 \circ F^1 = G^2 \circ F^2$.  Consider the pushout
    \[
      \begin{tikzcd}
        \tilde U
        \ar[d,"F^1"']
        \ar[r,"F^2"]
        &
        U^2
        \ar[d,"\bar G^2",dotted]
        \\
        U^1
        \ar[r,"\bar G^1"',dotted]
        &
        \bar U
      \end{tikzcd}
    \]
    Then, by~\Cref{lem:polyplex-comp},
    $(\bar U,\bar G^1(\bar u_1) \pcomp_i \bar G^2(\bar u_2))$ is a polyplex, and
    the universal property of pushouts gives a polyplex lifting
    \[
      H \co (\bar U,\bar G^1(\bar u_1) \pcomp_i \bar G^2(\bar u_2)) \to (\P,u)
      \zbox.
    \]
  \item If $u = \alpha$ for some generator $\alpha \in \P$, by induction, there
    are polyplex liftings
    \[
      G^\eps \co (U^\eps,u^\eps) \to (\P,\csrctgt\eps(u))
    \]
    for $\eps \in \set{-,+}$. By induction on~\Cref{prop:primitive-lifting},
    there exists a polyplex lifting
    \[
      F^{--} \co (S,s) \to (U^-,\csrc(u^-))
      \zbox.
    \]
    Similarly, there is a polyplex lifting of $(U^+,\csrc(u^+))$ and, since
    \[
      \freecat {(G^-)}(\csrc(u^-)) = \freecat{(G^+)}(\csrc(u^+))
    \]
    by~\Cref{prop:polyplex-lifting-unique}, it can be chosen to be of the form
    \[
      F^{+-} \co (S,s) \to (U^+,\csrc(u^+))
      \zbox.
    \]
    Similarly, there are polyplex liftings
    \[
      F^{-+} \co (T,t) \to (U^-,\ctgt(u^-))
      \qqtand
      F^{++} \co (T,t) \to (U^+,\ctgt(u^+))
      \zbox.
    \]
    Writing $F^\eps$ for $[F^{\eps-},F^{\eps+}]$ for $\eps \in \set{-,+}$,
    consider the pushout
    \[
      \begin{tikzcd}
        S \sqcup T
        \ar[r,"F^+"]
        \ar[d,"F^+"']
        &
        U^+
        \ar[d,"\bar G^+",dotted]
        \\
        U^-
        \ar[r,"\bar G^-"',dotted]
        &
        \partial U
      \end{tikzcd}
    \]
    and write $U$ for the \opol obtained from $\partial U$ by adding a
    generator
    \[
      \bar \alpha : \bar G^-(u^-) \to \bar G^+(u^+)
    \]
    (this is well-defined, since the definition of $\partial U$ ensures that
    $\csrctgt\eps(\bar G^-(u^-)) = \csrctgt\eps(\bar G^+(u^+))$ for
    $\eps \in \set{-,+}$). By \Cref{lem:polyplex-gen}, $(U,\bar \alpha)$ is a
    polyplex, and we easily deduce a polyplex lifting
    $H \co (U,\bar \alpha) \to (\P,\alpha)$.\qedhere
  \end{itemize}
\end{proof}

\noindent
Finally, we have the following uniqueness property of polyplex liftings:

\begin{prop}
  \label{prop:polyplex-lifting-unique}
  Given two morphisms~$L^1\co (U^1,u_1) \to (\P,u)$ and~$L^2\co (U^2,u_2) \to (\P,u)$
  of~$\Elt(\oPolb)$ where both~$(U^1,u_1)$ and~$(U^2,u_2)$ are primitive, there is an
  isomorphism~$\Theta\co (U^1,u_1) \to (U^2,u_2)$ such that~$L^2 \circ \Theta = L^1$.
\end{prop}
\begin{proof}
  We reason by case analysis on a formula for~$u$.
  \begin{itemize}
  \item If $u = \unit {u'}$, then the conclusion follows from induction hypothesis on~$u$.
  \item If $u = \alpha$ for some $\alpha \in \P$, then,
    by~\Cref{lem:polyplex-gen}, $U^1$ and $U^2$ are obtained by adding
    respective top-level generators $\alpha^1$ and $\alpha^2$ to polygraphs
    $\partial U^1$ and $\partial U^2$, the latter being expressed as pushouts
    \[
      \begin{tikzcd}[column sep=20mm]
        S^i \sqcup T^i
        \ar[r,"{[F^{i,+-},F^{i,++}]}"]
        \ar[d,"{[F^{i,--},F^{i,-+}]}"']
        &
        U^{i,+}
        \ar[d,"\bar G^{i,+}",dotted]
        \\
        U^{i,-}
        \ar[r,"\bar G^{i,-}"',dotted]
        &
        \partial U^i
      \end{tikzcd}
    \]
    for some principal $(S^i,s^i)$, $(T^i,t^i)$ and some primitive
    $(U^{i,-},u^{i,-})$, $(U^{i,+},u^{i,+})$ for $i \in \set{1,2}$ as in the
    statement of that lemma. In particular, $(U^{i,\eps},u^{i,\eps})$ are
    polyplex liftings of $\csrctgt\eps(\alpha)$ for $i \in \set{1,2}$ and~$\eps
    \in \set{-,+}$. By induction hypothesis, for $\eps \in \set{-,+}$, there are
    isomorphisms $\Theta^\eps \co (U^{1,\eps},u^{1,\eps}) \to
    (U^{2,\eps},u^{2,\eps})$. Since $(S^1,s^1)$ and $(T^1,t^1)$ are principal,
    we can easily verify with \Cref{lem:principal-parallel-eq} that
    \[
      G^{2,-} \circ \Theta^- \circ [F^{1,--},F^{1,-+}] = G^{2,+} \circ \Theta^+
      \circ [F^{1,+-},F^{1,++}]
    \]
    so that we get a morphism $\partial\Theta \co \partial U^1 \to \partial
    U^2$, which extends to a morphism $\Theta \co U^1 \to U^2$ such that
    $\Theta(\alpha^1) = \alpha^2$. Symmetrically, a morphism $\Theta' \co
    (U^2,\alpha^2) \to (U^1,\alpha^1)$ can be built.
    Using~\Cref{lem:principal-parallel-eq}, we easily verify that $\Theta$ and
    $\Theta'$ are inverse of each other.
  \item If $u = u_1 \pcomp_i u_2$, we use the pushout description
    from~\Cref{lem:polyplex-gen} and this case is then handled just like the
    previous one.\qedhere
  \end{itemize}
\end{proof}
\begin{rem}
  \label{rem:polconcrb-fam-rep}
  %
  A consequence of the existence and unicity properties above, together with
  \Cref{lem:principal-parallel-eq}, is that the functor
  $\polconcrb-\co \nPol\omega \to \Set$ of \Cref{ex:polconcrb} is
  \emph{familially representable}~\cite{carboni1995connected}, \ie can be
  expressed as a functor of the form
  \[
    \bigsqcup_{i \in I} \Hom(U^i,-) \co \oPol \to \Set\zbox.
  \]
  Here, $I$ is a set of representatives $(U^i,u^i)$ of all polyplexes
  (considered up to isomorphism of elements in $\Elt(\oPolb)$) of any
  dimension. Those can for instance be enumerated by constructing one polyplex
  liftings for each cell of the free precategory on the terminal polygraph. A
  similar description holds for the functor $\freecat -_k$, mapping a polygraph
  to the set of $k$\cells of the associated free precategory: the family $I$ is
  now a set of representatives for the polyplexes of dimension $k$ up to
  isomorphism.
\end{rem}

\begin{rem}
  A consequence of the canonicity of a polyplex liftings given by the above
  properties is that one can define a \eq{polyplex measure} on the cells of free
  precategories. Let $\P \in \oPol$, and write $\Z\P$ for the free $\Z$-module
  on $\poltoset\P$. Given $u \in \freecat\P$, one can define $\meas[\P](u)$ as
  follows. Consider a polyplex lifting $F \co (V,v) \to (\P,u)$ and define
  $S_V \in \Z V$ by $S_V = \sum_{g\in V} g$. Then, one defines $\meas[\P](u)$ as
  $\Z F(S_V)$. The definition of $\meas[\P](u)$ does not depend on the choice of
  $(V,v)$ by \Cref{lem:principal-parallel-eq,prop:polyplex-lifting-unique}. The
  question of the existence of a similar measure for free strict categories was
  raised by Makkai in~\cite{makkai2005word}. Later, using the standard
  Eckmann--Hilton for strict categories, the non-existence of such a measure was
  proven~\cite[Proposition 2.5.2.13]{forest:tel-03155192}.
\end{rem}



\section{Polygraphs as a presheaf category}
\label{sec:presheaf}
We can now use the results of the previous section in order to conclude that
$\oPol$ is a (concrete) presheaf category on the base category (also called
shape category) of \emph{plexes}, which are the elementary shapes polygraphs are
made of. In addition to the works of Burroni~\cite{burroni2012automates} and
Henry~\cite{henry2018nonunital}, this notion was also studied by
Makkai~\cite{makkai2005word} under the name ``computopes''.

Formally, a \emph{plex} is an element $(\P,u)\in\Elt(\oPol)$ which is primitive
(for the concrete structure introduced in \Cref{ex:pol-concrete}). Given an
element $(\Q,v)$ in $\Elt(\oPol)$, a \emph{plex lifting} is the data of a plex
$(\P,u)$ and a morphism of elements $F \co (\P,u) \to (\Q,v)\in\Elt(\oPol)$.

In order to relate the properties of plexes to the ones of polyplexes proved in
the previous section, we first need to briefly discuss the link between
$\Elt(\oPol)$ and $\Elt(\oPolb)$. We write~$\poltopolb\co \Elt(\oPol) \to
\Elt(\oPolb)$ for the canonical embedding. First note that, as a consequence of
\Cref{prop:freecat-functors}\ref{prop:freecat-functors:gen}, that
\begin{lem}
  \label{lem:U-fully-faithful}
  The functor $\mfunctorb U$ is fully faithful.
\end{lem}

\noindent
We then have the following.

\begin{prop}
  \label{prop:embedding-pol-polb}
  Let~$(\P,g) \in \Elt(\oPol)$. Then
  \begin{enumerate}[(1)]
  \item \label{prop:embedding-pol-polb:princ} $(\P,g)$ is principal if and only
    if~$\mfunctorb U(\P,g)$ is principal,
    
  \item \label{prop:embedding-pol-polb:plex} $(\P,g)$ is a plex if and only if~$\mfunctorb U(\P,g)$ is a polyplex.
  \end{enumerate}
\end{prop}
\begin{proof}
  By \Cref{lem:principal-charact},~\ref{prop:embedding-pol-polb:princ} holds.
  Suppose now that both~$(\P,g)$ and~$\mfunctorb U(\P,g)$ are principal. By
  \Cref{prop:freecat-functors}\ref{prop:freecat-functors:gen},~$\mfunctorb U$ is
  fully faithful, so that it reflects isomorphisms. Thus, if~$\mfunctorb
  U(\P,g)$ is a polyplex, then~$(\P,g)$ is a plex. For the converse, note that
  if~$f\co (\Q,v) \to \mfunctorb U(\P,g)$ is a morphism of~$\Elt(\oPolb)$,
  then, by \Cref{prop:freecat-functors}\ref{prop:freecat-functors:gen},~$v \in
  \Q$, so that
  \[
    (\Q,v)
    = \mfunctorb U(\Q,v)
    \qtand
    f = \mfunctorb U(f)\zbox.
  \]
  Hence, if~$(\P,g)$ is a plex, then~$\mfunctorb U(\P,g)$ is a polyplex.
\end{proof}

\begin{theo}
  \label{thm:opol-concr-ps}
  The category $\oPol$ is a concrete presheaf category.
\end{theo}
\begin{proof}
  We verify that the various conditions of Makkai's criterion
  (\cref{thm:concrete-charact}) are satisfied.
  \begin{enumerate}
  \item[\ref{thm:concrete-charact:iso}] Clear from the definition of $\oPol$.
  \item[\ref{thm:concrete-charact:colimits}] A consequence of general properties
    satisfied by categories of polygraphs derived from a globular monad (see
    Propositions 1.3.3.7 and 1.3.3.15 of \cite{forest:tel-03155192}).
  \item[\ref{thm:concrete-charact:small}] Since a primitive element~$(\P,g)$ is
    principal, the polygraph~$\P$ is finite. Thus, up to isomorphism, the sets
    $\P_i$ can be assumed to be subsets of $\N$. So
    that~\ref{thm:concrete-charact:small} holds.
  \item[\ref{thm:concrete-charact:prim-exists}] Given an element
    $(\P,g) \in \Elt(\oPol)$, by \Cref{prop:primitive-lifting}, there exists a
    polyplex lifting $F \co (U,u) \to (\P,g)$ of $\mfunctorb U (\P,g)$. By
    \Cref{prop:freecat-functors}\ref{prop:freecat-functors:gen}, we have that
    $u \in U$. Moreover, by
    \Cref{prop:embedding-pol-polb}\ref{prop:embedding-pol-polb:plex}, we have
    that $(U,u)$ is a plex, so~\ref{thm:concrete-charact:prim-exists} holds.
  \item[\ref{thm:concrete-charact:prim-equal}] Given
    $f,g\co (U,u) \to (X,x) \in \Elt(\oPol)$ with $(U,u)$ a primitive plex, then
    we have that
    $\mfunctorb U (f),\mfunctorb U(g)\co (U,u) \to (X,x) \in \Elt(\oPolb)$, so
    that $\mfunctorb U (f),\mfunctorb U(g)$ by \Cref{lem:principal-parallel-eq},
    and $f = g$ by faithfulness.
  \item [\ref{thm:concrete-charact:prim-iso}] Given two
    morphisms~$f\co (U,u) \to (X,x)$ and~$g\co (V,v) \to (X,x)$ of~$\Elt(\oPol)$
    where both~$(U,u)$ and~$(V,v)$ are primitive, we have by
    \Cref{prop:polyplex-lifting-unique} that there is an isomorphism
    \[
      \theta\co \poltopolb (U,u) \to \poltopolb (V,v) \in \Elt(\oPolb)
    \]
    such that~$\poltopolb (g) \circ \theta = \poltopolb (f)$. We conclude by the
    full faithfulness of $\poltopolb$.\qedhere
  \end{enumerate}
\end{proof}
\begin{rem}
  Following Makkai's proof of \cite[Theorem 4]{makkai2005word}, the base
  category of the presheaf category given by the above theorem is a small full
  subcategory of $\oPol$, whose objects are (the underlying polygraphs of)
  plexes, and such that every (underlying polygraph of a) plex is isomorphic to
  exactly one object of this subcategory. The objects of the latter can thus be
  easily enumerated, since they are in correspondence with the generators of the
  terminal polygraph~$\polterm$, as plex liftings.
\end{rem}

\begin{rem}
  Like the familial representability observed in \Cref{rem:polconcrb-fam-rep},
  the conditions
  \Cref{thm:concrete-charact:prim-exists,thm:concrete-charact:prim-equal,thm:concrete-charact:prim-iso}
  proved above entails a familial representability for the functor
  $\poltoset{-}$ of \Cref{ex:pol-concrete}, which can be expressed as
  \[
    \bigsqcup_{i\in I} \Hom(U^i,-) \co \oPol \to \Set
    \zbox.
  \]
  Here, $I$ is a set of representatives $(U^i,g^i)$ of all plexes (considered up
  to isomorphism of $\Elt(\oPol)$). By taking $I$ to be a set of representatives
  of all plexes of dimension $k$ for some $k \in \N$, one get a familial
  representability of the functor $(-)_k \co \oPol \to \Set$.
\end{rem}

\begin{rem}
  In~\cite{araujo2022computads}, \citeauthor{araujo2022computads} relies on
  \cite[Proposition 5.14]{dean2022computads}, which gives sufficient conditions
  for a category to be a presheaf category on a given full subcategory. The
  difference with \cite[Theorem~4]{makkai2005word} is that the latter is
  relative to a concrete presheaf structure, and is able to characterize the
  shape category as a full subcategory of primitive elements.
\end{rem}

\section{Parametric adjunction and genericic factorization}
\label{sec:parametricity}

While \Cref{rem:polconcrb-fam-rep} asserts that the cells of free precategories
on polygraphs are instances of \eq{universal shapes} (\ie polyplexes), a more
conceptual and general syntactical result can be given, which encompasses both
the existence of those universal shapes and the Conduché property of free
functors. This result relies on the existence of a parametric adjunction and an
associated generic factorization for the free functor $\freecat-$. Parametric
adjunctions and generic factorizations appear frequently in the context of
algebraic higher
category~\cite{street2000petit,weber2004generic,henry2018regular}: for example,
the free \ocat monad functor on globular sets is parametric right adjoint, and
has an associated generic factorization. While the classical parametric right
adjoints are monad functors on presheaf categories (for which characterization
criteria have been developed, for example~\cite[Theorem
2.13]{weber2007familial}), the unusual fact here is that the parametric right
adjoint $\freecat-$ is a left adjoint, whose codomain is not a presheaf
category, but the category of $n$\precategories: for us, this fact reflects and
summarize the good syntactical properties of the theory of precategories.

\newcommand\defslicecat[1][]{\ifempty{#1}{\oPCat}{\nPCat {#1}} / \freecat\polterm}%
While parametric adjunctions can easily be deduced from familial
representability properties (like \Cref{rem:polconcrb-fam-rep}) in a presheaf
setting (see \cite[Proposition 2.10]{weber2007familial}), there is no direct
criterion in our setting, so that we have to show the parametric
representability \eq{by hand}: we need to show that the functor
$\freecat-_\polterm \co \oPol \to \defslicecat$ is a right adjoint, where
$\defslicecat$ is the slice category of $\oPCat$ over the free precategory on
the terminal polygraph $\polterm$, and $\freecat-_\polterm$ the functor induced
by $\freecat-$. Since both $\oPol$ and $\defslicecat$ are locally presentable
categories, and that $\freecat-$ is a left adjoint, we are only required to show
that $\freecat-_\polterm$ preserves limits (see \cite[Theorem
1.66]{adamek1994locally}). Since $\oPol$ has a terminal object and the
computation of limits in $\defslicecat$ amounts to the computation of a
connected limit in $\oPCat$, we simply need to show that connected limits are
preserved, recovering \cite[Theorem 2.13]{weber2007familial} in our context. In
the following, given $k,n \in \N$, we write $D^{k,n}$, or simply $D^k$ for the
free $n$\precategory with one non-identity $k$\cell.
\begin{prop}
  \label{prop:freecat-pres-conn-limits}
  Given $n \in \Ninf$, the functor $\freecat - \co \nPol n \to \nPCat n$
  preserves connected limits.
\end{prop}
\begin{proof}
  First note that the functor $\cattoset- \co \nPCat n \to \Set$ is
  conservative; it is moreover familially representable by the $D^k$'s for $k <
  n+1$ (a $k$\cell of an $n$\precategory $C$ is the same thing as a functor $D^k
  \to C$) and thus preserves connected limits by \cite[Theorem
  2.5]{carboni1995connected} and \cite[Corollary 2.45]{adamek1994locally}. Since
  $\nPCat n$ is complete, it is sufficient to show that the functor
  $\polconcrb-\co \nPol n \to \Set$ preserves connected limits. But this functor
  is familially representable by \Cref{rem:polconcrb-fam-rep}, so that it
  preserves connected limits by \cite[Theorem 2.5]{carboni1995connected}.
\end{proof}
\noindent By the argument exposed earlier, we can conclude that:
\begin{theo}
  Given $n \in \Ninf$, the functor $\freecat-_\polterm \co \nPol n \to \defslicecat[n]$ is a right
  adjoint. In other words, $\freecat-$ is a parametric right adjoint.
\end{theo}
As a consequence, we have a generic factorization for the functor $\freecat-$.
We recall from~\cite{weber2004generic} the notion of generic morphism in the
present case: given $C \in \nPCat n$ and $\P \in \nPol n$, a morphism $F \co C
\to \freecat\P$ is \emph{generic} when, for any commutative square of the form
\[
  \begin{tikzcd}
    C
    \ar[r,"G"]
    \ar[d,"F"']
    &
    \freecat\Q
    \ar[d,"\freecat H"]
    \\
    \freecat\P
    \ar[r,"\freecat K"']
    \ar[ur,dotted,"\freecat L"description]
    &
    \freecat\R
  \end{tikzcd}
\]
for some $\Q,\R \in \nPol n$, $G \co C \to \freecat\Q$ in $\nPCat n$, $H \co \Q
\to \R$ and $K \co \P \to \R$ in $\nPol n$, there exists a unique $L \co \P \to
\Q$ such that $G = \freecat L \circ F$ and $K = H \circ L$. Now, given a
morphism $F \co C \to \freecat \P$, a \emph{generic factorization} is a
decomposition of $F$ as $\freecat H \circ G$ for some $\Q \in \nPol n$, some
generic $G \co C \to \freecat\Q$ and $H \co \Q \to \P \in \nPol n$. By the
universal property of generic morphisms, such a decomposition is unique up to an
isomorphism $\Q \to \Q'$.

\begin{coro}
  Given $n\in \Ninf$, $C \in \nPCat n$ and $\P \in \nPol n$, every $F \co C \to
  \freecat\P \in \nPCat n$ admits a generic factorization.
\end{coro}
\begin{proof}
  By \cite[Proposition 2.6]{weber2007familial}, the existence of generic
  factorizations follows from the fact that $\freecat-_\polterm$ is a parametric
  right adjoint.
\end{proof}
\noindent Some generic morphisms are easy to identify:
\begin{prop}
  \label{prop:polyplex-generic}
  Given $n \in \N$ and $\P \in \oPol$, writing $u$ for the non-identity $n$\cell
  of $D^n$, a functor $F \co D^n \to \freecat\P$ is generic if and only if
  $(\P,F(u))$ is a polyplex.
\end{prop}
\begin{proof}
  We start with the first implication. Let $H \co (\Q,v) \to (\P,F(u))$ be a
  polyplex lifting of $(\P,F(u))$. Then, writting $G \co D^n \to \freecat\Q$
  sending $u$ to $v$, we have $\freecat H \circ G = \freecat{(\unit{\P})} \circ
  F$. Thus, there exists a unique lifting $L \co \P \to \Q$ such that $\freecat
  L \circ F = G$ and $H \circ L = \unit{\P}$. In particular, we have that
  $\freecat L(F(u)) = v$ and $L$ is a monomorphism. Thus, since $(\Q,v)$ is
  principal, $L \co (\P,F(u)) \to (\Q,v)$ is an isomorphism.

  Conversely, let 
  \[
    \begin{tikzcd}
      D^n
      \ar[r,"G"]
      \ar[d,"F"']
      &
      \freecat\Q
      \ar[d,"\freecat H"]
      \\
      \freecat\P
      \ar[r,"\freecat K"']
      &
      \freecat\R
    \end{tikzcd}
  \]
  be a commutative square where $(\P,F(u))$ is assumed to be a polyplex.
  Consider a polyplex lifting $L \co (\bar \P,\bar u) \to (\Q,G(u))$. By
  applying $\freecat H$, $(\bar \P,\bar u)$ is a polyplex lifting of $\freecat
  H(G(u)) = \freecat K(F(u))$ and so is $(\P,F(u))$. By
  \Cref{prop:polyplex-lifting-unique}, we may assume $(\bar \P,\bar u) =
  (\P,F(u))$ with $H \circ L = K$. Moreover, since $\freecat L(F(u)) = G(u)$, we
  have $L \circ F = G$ by freeness of $D^n$. Finally, the unicity of the lifting
  $L$ of the above square is a consequence of \Cref{lem:principal-parallel-eq}.
\end{proof}
\begin{rem}
  In a related manner, given $n \in \N$ and $v \in \freecat\polterm_n$, the
  image of $D^n \xto{v} \freecat\polterm$, seen as an object of $\defslicecat$,
  by a left adjoint to $\freecat-_\polterm$ is the underlying polygraph of a
  polyplex lifting of $v$.
\end{rem}
\begin{rem}
  The above generic factorization can be seen as a stronger version of
  \Cref{prop:conduche}. Indeed, given $k,l > 0$ and $i = \min(k,l) - 1$, there
  exists a polygraph $\D^{k,l}$ such that $\freecat{(\D^{k,l})}$ is the free
  \opcat with one $k$\cell $u_1$ and one $l$\cell $u_2$, such that $\ctgt_i(u_1)
  = \csrc_i(u_2)$. The construction of $\D^{k,l}$ can be seen to induce a
  polyplex $(\D^{k,l},u_1\comp_i u_2)$ by \Cref{lem:polyplex-comp}. Writting $n$
  for $\max(k,l)$ and $F^{k,l} \co D^n \to \freecat{(\D^{k,l})}$ for the functor
  sending the non-trivial $n$\cell of $D^n$ to $u_1\comp_i u_2$,
  \Cref{prop:conduche} amounts to observe that the $F^{k,l}$'s are generic by
  \Cref{prop:polyplex-generic}.
\end{rem}


\section{Toward homotopical properties of precategories}
\label{sec:folk}

In this section, we report on failed attempts to study homotopical properties of
categories, leaving open questions for future works.

\paragraph{A folk model structure on precategories?}
In the setting of strict $n$-categories, the usefulness of polygraphs can be
explained by the facts that they are free objects such that every category
admits a description by such an object, and any two descriptions are suitably
equivalent. In more precise and modern terms, this was formalized by Lafont,
Métayer and Worytkiewicz~\cite{lafont2010folk}, who constructed a structure of
model category on the category $\nCat\omega$ of strict $\omega$-categories, in
which weak equivalences are the expected equivalences of $\omega$-categories and
cofibrant objects are $\omega$-categories freely generated by polygraphs.
One could expect that we could perform a similar construction on precategories,
and construct a model structure where weak equivalences are the expected ones
and cofibrant objects are polygraphs in the sense of this article. Whether this
is possible or not is left as an open question, but explain here that a direct
adaptation of the proof of~\cite{lafont2010folk} does not go through easily.

Let us first introduce some terminology. Given an \opcat~$C$, we make the
following coinductive mutual definitions:
\begin{itemize}
\item two cells $x,y \in C$ of the same dimension are \textit{equivalent},
  denoted $x\sim y$, when there exists an equivalence $u : x \to y$;
\item a cell $u : x \to y$  is an \textit{equivalence} when there exists $\bar
  u : y \to x$ such that $u \comp \bar u \sim \unit x$ and $\bar u \comp u
  \sim \unit y$.
\end{itemize}
We could then have hope for the following definition of weak equivalences. Given
an \oprefunctor $F : C \to D$, $F$ is a \textit{weak equivalence} when it is
``essentially surjective in every dimension'', \ie
\begin{itemize}
\item for every $0$\cell $y \in D_0$, there exists $x \in C_0$ such that
  $Fx \sim y$,
\item for every pair of parallel cells $u,u' \in C$ and cell
  $\bar v : F(u) \to F(u')$, there exists $v \in C$ such that
  $F(v) \sim \bar v$.
\end{itemize}

The above definitions directly generalize the ones for strict categories. The
construction of the folk model structure on strict \ocats then requires a
\emph{weak division property}~\cite[Lemma~5]{lafont2010folk}, which the authors
present as being ``crucial''. The direct generalization of it in the setting of
precategories is as follows:

\begin{property}[Weak division]
  \label{weak-division}
  Given an \opcat $C$ and an equivalence $u : x \to y \in C_1$, for any
  $1$\cells $s,t : y \to z$ and for any $2$\cell
  $w : u \comp_0 s \To u \comp_0 t$,
  \begin{enumerate}[(a)]
  \item \label{wd-e} there is a $2$\cell $v : s \To t$ such that
    $u \comp_0 v \sim w$,
    \[
      \begin{tikzcd}[row sep=1ex]
        &y\ar[dr,bend left,"s"]\\
        x\ar[ur,bend left,"u"]\ar[dr,bend right,"u"']\ar[rr,phantom,"w\Downarrow"]&&z\\
        &y\ar[ur,bend right,"t"']
      \end{tikzcd}
      \qquad\sim\qquad
      \begin{tikzcd}
        x\ar[r,"u"]&y\ar[r,bend left=50,"s"]\ar[r,bend right=50,"t"']\ar[r,phantom,"v\Downarrow"]&z
      \end{tikzcd}
    \]
  \item \label{wd-u} for any $2$\cells $v,v':s\To t$ such that
    $u\comp_0v\sim w\sim u\comp_0 v'$ we have $v\sim v'$.
  \end{enumerate}
\end{property}

\noindent
We would also need a generalization of the above property for $n$-cells, but we
will see that the proof of the stated property in dimension~1 already fails to
generalize from strict categories to precategories. Consider cells $u$ and $w$
as in the above property, with $u$ reversible, and let us try to define the
cell~$v$. Writing $r:\bar u\comp_0 u\to\id x$ for the $2$-cell witnessing that
$u$ is reversible, following~\cite{lafont2010folk}, we are tempted to define~$v$
as
\[
  v=\ol{(r\comp_0 s)}\comp_1(\bar u\comp_0w)\comp_1(r\comp_0t)
\]
If we picture $r$ and $w$ as on the left, $v$ can be pictured as on the right:
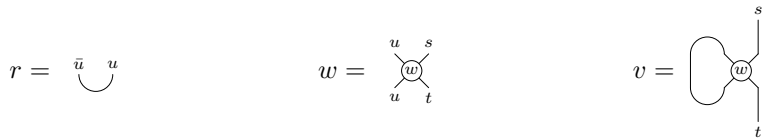
\begin{align*}
  r&=\satex{r}
  &
  w&=\satex{w}
  &
  v&=\satex{v}
\end{align*}
In particular, in the case where $w$ is of the form $w=u\comp_0 v'$ for some
$2$-cell $v':s\To t$, we should have $v\sim v'$ by \ref{wd-u}. In the case of
strict categories, this holds thanks to the interchange law:
\[
  \begin{array}{r@{\ =\ }c@{\ =\ }c@{\ \sim\ }c}
    v
    &
    \ol{(r\comp_0 s)}\comp_1(\bar u\comp_0u\comp_0 v')\comp_1(r\comp_0t)
    &
    \ol{(r\comp_0 s)}\comp_1(r\comp_0s)\comp_1v'
    &
    v'
    \\
    \satex{ex1a}
    &
    \satex{ex1b}
    &
    \satex{ex1c}
    &
    \satex{ex1d}
  \end{array}
\]
However, in the case of precategories there is no reason why this should hold.
Of course, this does not directly imply that \cref{weak-division} does not hold
or that there is no suitable model structure on precategories, but more work is
required than a mere adaptation of~\cite{lafont2010folk}. The above also suggests
that it could be interesting to investigate structures ``in between'' precategories
and strict categories, where the interchange law is only required to hold for
some morphisms (such as $r$ in the above example).


\paragraph{A cone construction?}
Another homotopy-related question one might ask is whether the underlying shape
category of the presheaf category of polygraphs of precategories is able to
model homotopy types. A now standard approach to get a positive answer is to
show that this shape category is a \emph{weak test
  category}~\cite{grothendieck2021pursuing,maltsiniotis2005theorie}, \ie a
category $C$ whose presheaf category $\hat C$ can be equipped with a canonical
class of weak equivalences $\cW$, such that the induced localization $\hat
C[\finv\cW]$ is canonically isomorphic to the homotopy category~$\Hot$, so that,
in particular, $\hat C$ models all homotopy types.

A common way to show that a category is a weak test category is to exhibit a
\emph{separating décalage}~\cite{maltsiniotis2005theorie} on this category.
Formally, a \emph{décalage} on a catégorie $C$ is given by a functor $D : C
\to C$ together with natural transformations
\[
  \begin{tikzcd}
    1_C
    \ar[r,"\alpha"]
    &
    D
    &
    \ar[l,"\beta"']
    \top
  \end{tikzcd}
\]
where $\top$ is an object of $C$ seen as a constant functor. Such a décalage is
\emph{separating} when we moreover have that
\begin{enumerate}[(a)]
\item for every $c \in C_0$, the arrow $\alpha_c : c \to D(c)$ is a monomorphism,
\item $\alpha$ is \emph{cartesian}: for every morphism $f : c \to c' \in C$, the
  diagram
  \[
    \begin{tikzcd}
      c
      \ar[r,"\alpha_c"]
      \ar[d,"f"']
      &
      D(c)
      \ar[d,"D(f)"]
      \\
      c'
      \ar[r,"\alpha_{c'}"']
      &
      D(c')
    \end{tikzcd}
  \]
  is a pullback,
\item for every $c \in C_0$, there is no commutative diagram of the form
  \[
    \begin{tikzcd}
      c'
      \ar[r,"g"]
      \ar[d,"f"']
      &
      \top
      \ar[d,"\beta_c"]
      \\
      c
      \ar[r,"\alpha_c"']
      &
      D(c)
    \end{tikzcd}
  \]
  for some $c' \in C_0$ and $f\co c' \to c$ and $g\co c' \to \top$ in $C_1$.
\end{enumerate}
Following Henry's line of proof for the case of regular
plexes~\cite{henry2018regular}, a promising choice of décalage in a polygraphic
setting is the one where $D$ is \eq{cone construction} functor, also called
\emph{expansion functor}: starting from a polygraph $\P$, this functor adds to
$\P$ a $0$\generator $o$, a $1$\generator $x \to o$ for each $x \in \P_0$, and
more generally an $(i{+}1)$\generator for each $i$\generator of $\P$, so that
$D\P$ appears as a combinatorial description of a cone over the \eq{space} defined
by $\P$. Then, continuing the definition of a décalage, one can take $\alpha$ to
be the canonical embedding of a polygraph into the base of its cone, $\top$ to
be the polygraph with only one $0$\generator~$o$, and $\beta$ to be the marking
of $o$ as the top of each constructed cone.


While Henry~\cite{henry2018regular} used the join of strict
categories~\cite{ara2016joint} to define the expansion functor on regular
plexes, a more direct description of this construction was used by
\citeauthor*{ara2022orientals}~\cite{ara2022orientals} in the case of strict
categories that we unsuccessfully tried to adapt to precategories. In the
following, we describe this attempt, hoping it can still benefit other settings.
Write $\oPCatp$ for the category of \emph{pointed \opcats}, that is, the
category whose objects are the pairs $(C,o)$ where $o \in C_0$ and the morphisms
$(C,o_C) \to (D,o_D)$ are the functor $F\co C \to D$ such that $F(o_C) = o_D$.
We have an evident adjunction
\begin{equation}
  \label{eq:cone-first-adj}
  \begin{tikzcd}
    \oPCat
    \ar[r,bend left,"\pointedfreef-"]
    \phar[r,"\bot"]
    &
    \oPCatp
    \ar[l,bend left,"\pointedfgf"]
  \end{tikzcd}
\end{equation}
where $\pointedfgf$ simply forgets the pointed $0$\cell $o$. In order to define
an expansion functor on precategories, one wants to introduce a functor
\[
  \conef \co \oPCatp \to \oPCatp
\]
such that $\conef (C,o)$ is the \opcat of \emph{$i$\cones} on $(C,o)$ for $i \in
\N$: a $0$\cone is some \eq{base} $0$\cell $\cb x \in C$ together with some
$1$\cell $\cc x \co \cb x \to o$ of $C$, a $1$\cone between $(\cb x,\cc x)$ and
$(\cb y,\cc y)$ is a \eq{base} $1$\cell $\cb f \co x \to y$ and $\cc f \co \cb f
\comp_0 \cc y \To \cc x$, and so on. There is then a natural embedding
$\gamma_{(C,o)} \co \conef (C,o) \to (C,o)$, mapping every $i$\cone to its
\eq{base} $i$\cell. If such a functor exists, one could then define the category
of \emph{conic precategories} $\oPCatc$ whose objects are the triples $(C,o,\sigma)$,
where $(C,o)$ is a pointed \opcat and $\sigma$ is a section of $\gamma_{(C,o)}$
satisfying adequate degeneracy conditions (see \cite[Definition
2.2.1]{ara2022orientals}), and whose morphisms are the ones of $\oPCatp$ which
adequately commute with the sections. In other words, an object of $\oPCatc$ is
a pointed \opcat with the data of a compatible $i$\cone for every $i$\cell,
satisfying degeneracy conditions. The forgetful functor $\conicfgf\co\oPCatc \to
\oPCatp$ should then admit a left adjoint, so that we get an adjunction
\begin{equation}
  \label{eq:cone-second-adj}
  \begin{tikzcd}
    \oPCatp
    \ar[r,bend left,"\conicfreef-",pos=0.54]
    \phar[r,"\bot"]
    &
    \oPCatc
    \ar[l,bend left,"\conicfgf"]
  \end{tikzcd}
  \zbox.
\end{equation}
The expansion functor is then the functor
\[
  \tilde D = \pointedfgf \circ \conicfgf \circ \conicfreef- \circ \pointedfreefpar-
  \co \oPCat \to \oPCat
\]
which is the underlying functor of the monad of the composition of the two
adjunctions \eqref{eq:cone-first-adj} and~\eqref{eq:cone-second-adj}. Then, one
could show that this functor restricts well to polygraphs (just like for the case
of strict categories~\cite{ara2022orientals}), so that we get $D \co \oPol \to
\oPol$, and then show that $D$ is the underlying functor of a separating
décalage.

Sadly, the definition of $\conef$ does not go through for precategories. Given a
pointed \opcat $(C,o)$, even though one can follow the concrete definition of
\cite{ara2022orientals} to get a globular set $\conef(C,o)$ equipped with
precategorical compositions operations, one can show that the latter do not
satisfy axiom~\Cref{precat:distrib} of precategories in general: the lack of
interchange law for precategories is to blame here.

While it is not formally excluded that the shape category of plexes is a weak
test category, the fact that it does not admit an expansion functor while the
one of regular plexes does is already a bad sign which suggests, in addition to
the difficulty to define a notion of weak equivalences with good properties (as
discussed at the beginning of this section), that \eq{bare} precategories are
not an adequate tool for homotopical purposes (but that does not prevent them to
be used to define other adequate tools, like Gray
categories~\cite{forest2021rewriting}). Maybe this could be linked to the fact
that the underlying operad of precategories is not contractile, and should be
better understood in future work.

\paragraph*{Acknowledgments.}
The authors would like to thank Manuel Araújo for the discussions about their
shared interest on precategories and their applications. They also would like to
thank Léonard Guetta for his useful comments about this work, in particular on
the use of strict Conduché functors in a precategorical setting.

\printbibliography
\end{document}